\documentclass[psamsfonts,onesided,12pt]{amsart}
\usepackage{amsmath,amssymb, amsthm, graphicx,hyperref}
\usepackage[all,arc,cmtip]{xy}
\usepackage{enumerate}

\usepackage{tikz}
\usetikzlibrary{arrows,chains,matrix,positioning,scopes}
\makeatletter
\tikzset{join/.code=\tikzset{after node path={%
\ifx\tikzchainprevious\pgfutil@empty\else(\tikzchainprevious)%
edge[every join]#1(\tikzchaincurrent)\fi}}}

\makeatother
\tikzset{>=stealth',every on chain/.append style={join}, every join/.style={->}}

\usepackage{mathrsfs}
\usepackage[margin=1in]{geometry}

\numberwithin{equation}{section}

\theoremstyle{plain}% default
\newtheorem{theorem}[equation]{Theorem}
\newtheorem{lemma}[equation]{Lemma}
\newtheorem{proposition}[equation]{Proposition}
\newtheorem{corollary}[equation]{Corollary}

\theoremstyle{definition}
\newtheorem{definition}[equation]{Definition}

\newtheorem*{convention}{Convention}

\theoremstyle{remark}

\newtheorem*{note}{Note}

\newcommand{\bigslanted}[2]{{\left.\raisebox{.2em}{$#1$}\middle/\raisebox{-.2em}{$#2$}\right.}}

\newcommand{\mc}[1]{\mathcal{#1}}
\newcommand{\mf}[1]{\mathfrak{#1}}
\newcommand{\ms}[1]{\mathscr{#1}}
\newcommand{\ov}[1]{\overline{#1}}
\newcommand{\op}[1]{\operatorname{#1}}

\newcommand{\st}{\text{ } \big| \text{ }}

\newcommand{\inv}{^{-1}}

\newcommand{\Z}{\mathbb{Z}} \newcommand{\N}{\mathbb{N}}
 
\newcommand{\A}{\mathbb{A}}

\newcommand{\Hom}{\operatorname{Hom}}
\newcommand{\End}{\operatorname{End}}
\newcommand{\Ext}{\operatorname{Ext}}

\newcommand{\coker}{\operatorname{coker}}
\newcommand{\Rk}{\operatorname{Rk}}
\newcommand{\Soc}{\operatorname{Soc}}
\newcommand{\Ann}{\operatorname{Ann}}

\newcommand{\define}[1]{\emph{#1}}
\newcommand{\U}[1]{U(\mf{#1})}
\newcommand{\V}[1]{V(\mf{#1})}
\newcommand{\ev}[1]{#1 \mathrm{_{\overline{0}}}}
\newcommand{\od}[1]{#1 \mathrm{_{\overline{1}}}}
\newcommand{\supalg}[1]{\ev{#1} \oplus \od{#1}}
\newcommand{\sgndeg}[2]{(-1)^{| #1 | | #2 |}}
\newcommand{\sy}[2]{\Omega^{#1}(#2)}
\newcommand{\sL}{\mf{sl}(1|1)}
\newcommand{\gen}[1]{ \langle #1 \rangle}
\newcommand{\rel}[1]{( \mf{#1}, \ev{\mf{#1}})}
\newcommand{\ind}[1]{\U{#1} \otimes_{U(\ev{\mf{#1}})}}

\begin{document}
\title{%
On Endotrivial Modules for Lie Superalgebras }
\author{Andrew J. Talian }
\thanks{Research of the author was partially supported
by NSF grant DMS-0738586.}
\address{Department of Mathematics\\ University of Georgia \\
Athens\\ GA 30602, USA}
\curraddr{Department of Mathematics \\ Concordia College \\ Moorhead \\ MN 56562, USA}
\email{atalian@cord.edu} %\\
%Department of Mathematics\\
%University of Georgia\\
%Athens, GA 30602}
\date{June 2013}

%\email{atalian@math.uga.edu} %\\
%Department of Mathematics\\
%University of Georgia\\
%Athens, GA 30602}

\begin{abstract}
Let $\mf{g} = \supalg{\mf{g}}$ be a Lie superalgebra over an algebraically
closed field, $k$, of characteristic 0.  An endotrivial $\mf{g}$-module,
$M$, is a $\mf{g}$-supermodule such that $\Hom_k(M,M) \cong k_{ev} \oplus P$
as $\mf{g}$-supermodules, where $k_{ev}$ is the trivial module concentrated
in degree $\overline{0}$ and $P$ is a projective $\mf{g}$-supermodule.
In the stable module category, these modules form a group under the
operation of the tensor product.  We
show that for an endotrivial module $M$, the syzygies $\sy{n}{M}$
are also endotrivial, and for certain
Lie superalgebras of particular
interest, we show that $\sy{1}{k_{ev}}$ and the parity change functor
actually generate the group of endotrivials.  Additionally, for a
broader class of Lie superalgebras,
for a fixed $n$, we show that there are finitely many
endotrivial modules of dimension $n$.
\end{abstract}

\date{}
\maketitle

\section{Introduction}
The study of endotrivial modules began with Dade in 1978 when he defined
endotrivial $kG$-modules for a finite group $G$ in \cite{Dade1-1978} and
\cite{Dade2-1978}.  Endotrivial modules arose naturally in this context
and play an important role in determining
the simple modules for $p$-solvable groups.
Dade showed that, for an abelian $p$-group $G$,
endotrivial $kG$-modules   have
the form $\sy{n}{k} \oplus P$ for some projective module $P$, where
$\Omega^n(k)$ is the $n$th syzygy (defined in Section \ref{S:prelim})
of the trivial module $k$.  In general, in the stable module category,
the endotrivial modules form an abelian group under the
tensor product operation.
It is known, via Puig in \cite{Puig-1990}, that this group is finitely
generated in the case of $kG$-modules and is completely classified
for $p$-groups over a field of characteristic $p$
by Carlson and Th\'evenaz in \cite{CT-2004} and \cite{CT-2005}.  An
important step in this classification is a technique where the
modules in question are restricted to elementary abelian subgroups.

Carlson, Mazza, and Nakano have also computed the group of endotrivial
modules for finite groups of Lie type (in the defining characteristic)
in \cite{CMN-2006}.  The same authors in \cite{CMN-2009} and
Carlson, Hemmer, and Mazza in \cite{CHM-2010} give a classification
of endotrivial modules for the case when $G$ is either the symmetric or
alternating group.

This class of modules has also been studied for modules over finite
group schemes by Carlson and Nakano in \cite{CN-2009}.  The authors
show that all endotrivial modules for a unipotent abelian
group scheme have the form
$\sy{n}{k} \oplus P$ in this case as well.  For certain
group schemes of this type, a classification is also given in the same
paper (see Section 4).  The same authors proved, in an extension
of this paper, that given an arbitrary finite group scheme, for a fixed
$n$, the number of isomorphism
classes of endotrivial modules of dimension $n$ is finite (see
\cite{CN-2011}), but it is not known whether the endotrivial group is
finitely generated in this context.

We wish to extend the study of this class of modules to Lie superalgebra
modules.  First we must establish the correct notion of endotrivial
module in this context.
Let $\mf{g} = \supalg{\mf{g}}$ be a Lie superalgebra over an algebraically
closed field, $k$, of characteristic 0.
A $\mf{g}$-supermodule, $M$, is called endotrivial
if there is a supermodule isomorphism
$\Hom_k (M,M) \cong k_{ev} \oplus P$ where $k_{ev}$ is the
trivial supermodule concentrated in degree $\overline{0}$ and $P$ is
a projective $\mf{g}$-supermodule.

There are certain subalgebras, denoted $\mf{e}$ and $\mf{f}$,
of special kinds of classical Lie superalgebras which are of interest.
These subalgebras ``detect'' the cohomology of the Lie superalgebra
$\mf{g}$.  By this, we mean that the cohomology for $\mf{g}$ embeds
into particular subrings of the cohomology for $\mf{e}$ and $\mf{f}$.
These detecting subalgebras can be considered analogous to elementary
abelian subgroups and are, therefore, of specific interest.

In this paper, we observe that the universal enveloping
Lie superalgebra $U(\mf{e})$ has a very similar structure to the
group algebra $kG$ when $G$ is abelian, noncyclic of order 4 and
$\operatorname{char} k = 2$ (although $U(\mf{e})$ is not commutative).
With this observation, we draw
from the results of \cite{Carlson-1980} to prove the base case in an
inductive argument for the classification of the group of endotrivial
$U(\mf{e})$-supermodules.  The inductive step uses
techniques from \cite{CN-2009} to complete the classification.
For the other detecting subalgebra $\mf{f}$, even though
$U(\mf{f})$ is not isomorphic to $U(\mf{e})$, reductions are made to
reduce this case to the same proof.

The main result is that for the detecting subalgebras $\mf{e}$ and
$\mf{f}$, denoted generically as $\mf{a}$,
the group of endotrivial supermodules, $T(\mf{a})$, is isomorphic
to $\Z_2$ when the rank of $\mf{a}$ is one and $\Z \times \Z_2$ when
the rank is greater than or equal to two.

We also show that for a classical Lie superalgebra $\mf{g}$
such that there are finitely many simple $\ev{\mf{g}}$-modules of
dimension $\leq n$, there are only finitely many endotrivial
$\mf{g}$-supermodules of a fixed dimension $n$.  This is done by considering
the variety of $n$ dimensional representations as introduced by Dade in
\cite{Dade-1979}.  In particular, this result holds for classical
Lie superalgebras such that $\ev{\mf{g}}$ is a semisimple Lie superalgebra.

\section{Notation and Preliminaries} \label{S:prelim}
First, a few basic definitions are given.  Only a few are included
here and any others may be found in \cite[Chapters 0,1]{Scheunert-1979}.
In this paper, $k$ will always be an algebraically closed field of
characteristic 0 and a Lie superalgebra $\mf{g}$ will always be defined
over $k$.

Let $R$ be either the ring $\Z$ or $\Z_2$.  The structure of the
tensor product of two $R$-graded vector spaces is the usual one, however,
given $A$ and $B$, $R$-graded associative algebras, the vector space
$A \otimes B$ is an $R$-graded algebra under multiplication defined by
$$
(a \otimes b)(a' \otimes b') = \sgndeg{a'}{b} (aa') \otimes (bb')
$$
where $a$ and $b$ are homogeneous elements and the notation
$|a|$ indicates the degree in which $a$ is concentrated,
i.e. $a \in A_{|a|}$ for any homogeneous $a \in A$.  The definition is extended
to general elements by linearity.
This is called the \define{graded tensor product of $A$ and $B$} and is denoted by $A \overline{\otimes} B$.

\begin{convention} From now on, all elements are assumed to be homogeneous
and all definitions are extended to general elements by
linearity, as is done above.
\end{convention}

A \define{homomorphism} is always assumed to be a homogeneous map of
degree 0.  Additionally, any associative superalgebra $A$ can
be given the structure of a Lie superalgebra, which is denoted $A_L$,
by defining
\begin{equation} \label{E:bracket}
[a,b] := ab - \sgndeg{a}{b} ba.
\end{equation}
This construction applied to the associative 
superalgebra $\End_k V$, where $V = \supalg{V}$ is a super vector space
over $k$, yields $(\End_k V)_L$ which is the Lie superalgebra $\mf{gl}(V)$.

An important tool for dealing with the $\Z_2$ grading of the modules
in this setting is the functor which interchanges the grading of
the module but does not alter any other structure.
\begin{definition}
Let $\mf{g}$ be a Lie superalgebra.
Define a functor
$$
\Pi : \operatorname{mod}(\mf{g}) \rightarrow \operatorname{mod}(\mf{g})
$$
by setting
$\Pi(V) = \supalg{\Pi(V)}$ where $\ev{\Pi(V)} = \od{V}$ and $\od{\Pi(V)} = \ev{V}$ 
for any $\mf{g}$-module $V = \supalg{V}$.  If $\phi$ is a morphism
between $\mf{g}$-modules then $\Pi(\phi) = \phi$.
This operation is known as the \define{parity change functor}.
\end{definition}
Note that, since the vector spaces are the same, the endomorphisms
of $\Pi(V)$ are the same as $V$.
Thus, it is clear that if $V$ is a $\mf{g}$-supermodule,
then the same action turns $\Pi(V)$ into a $\mf{g}$-supermodule as well,
however, this module does not necessarily have to be isomorphic
to the original module.  An explicit construction of this functor is
given later (see Lemma \ref{L: Pi is Omega invariant}).

Now that $\Pi$ has been defined, one
particularly important module to consider is that of the trivial
module $k$, concentrated in degree $\ov{0}$.  In order to distinguish
the grading, this module will be denoted as $k_{ev}$ and
$k_{od} := \Pi(k_{ev})$.

At this point, we now wish to specialize to the category of interest
for the remainder of this paper.  Let $\mf{t} \subseteq \mf{g}$ be
a subalgebra and define $\mc{F}_{(\mf{g},\mf{t})}$
to be the full subcategory
of $\mf{g}$-modules where the objects are finite dimensional modules
which are completely reducible over $\mf{t}$.  Note
that this category has enough projectives and is self injective,
as detailed in \cite{BKN1-2006} and \cite{BKN3-2009} respectively.

Now, we can introduce a class of modules which will be of particular
interest for the remainder of this paper.  This definition
can be used in categories with enough projectives and enough injectives.

\begin{definition} \label{D:syzygy}
Let $\mf{g}$ be a Lie superalgebra and let $M$ be a $\mf{g}$-supermodule.  Let
$P$ be the minimal projective which surjects on to $M$ (called the projective
cover), with the map
$
\psi : P \twoheadrightarrow M.
$
The \define{first syzygy of $M$} is defined to be $\ker \psi$ and is
denoted $\Omega(M)$ or $\Omega^1(M)$.  This is also referred to as a
Heller shift in some literature.  Inductively, define
$\Omega^{n+1} := \Omega^1(\Omega^n)$.

Similarly, given $M$, let $I$ be the injective hull of $M$ with the
inclusion
$
\iota : M \hookrightarrow I,
$
then define $\Omega\inv (M) := \coker \iota$.  This is extended
by defining $\Omega^{-n-1} := \Omega\inv ( \Omega^{-n})$.

Finally, define $\Omega^0(M)$ to be the compliment of the largest projective direct
summand of $M$.  In other words, we can write $M = \Omega^0(M) \oplus Q$
where $Q$ is projective and maximal as a projective summand.
Thus, the \define{$n$th syzygy of $M$} is defined for any
integer $n$.
\end{definition}

\section{Endotrivial Modules}
\begin{note}  We are working in the category
$\mc{F} := \mc{F}_{(\mf{g}, \mf{\ev{g}})}$ and all modules are assumed
to be $\Z_2$-graded (i.e. supermodules).
\end{note} 

\begin{definition}
Given a category of modules, $\mc{A}$,
consider the category with the same objects
as the original category and an equivalence relation on the morphisms
given by $ f \sim g$ if $f - g$ factors through a projective module in
$\mc{A}$.
This is called the \define{stable module category} of $\mc{A}$ and is
denoted by $\op{Stmod}(\mc{A})$.
\end{definition}

\begin{definition}
Let $\mf{g}$ be a Lie superalgebra and $M$ be a $\mf{g}$-module.
We say that $M$ is
an \define{endotrivial module} if $\End_k (M) \cong k_{ev} \oplus P$ where 
$P$ is a projective module in $\mc{F}$.
\end{definition}
Since we have the module isomorphism $\Hom_k(V, W) \cong W \otimes V^*$
for two $\mf{g}$-modules $V$ and $W$, often times the condition
for a module $M$ being endotrivial is rewritten as
$$
M \otimes M^* \cong k_{ev} \oplus P.
$$

\begin{lemma}[Schanuel]
Let $0 \rightarrow M_1 \rightarrow P_1 \rightarrow M \rightarrow 0$ and
$0 \rightarrow M_2 \rightarrow P_2 \rightarrow M \rightarrow 0$ be short
exact sequences of modules where $P_1$ and $P_2$ are projective, then
$M_1 \oplus P_2 \cong M_2 \oplus P_1$.
\end{lemma}
The proof is straightforward
and can be found in \cite{Benson1-1998}.  This lemma
is useful because it shows that, in the stable category, it is not
necessary to use the projective cover to obtain the syzygy of
a module.  Indeed, any projective module will suffice because the
kernel of the projection maps will only differ by a projective summand.

\begin{proposition} \label{P:selfinj}
The category $\mc{F}$ is self injective.  That is, a module $M$ in
$\mc{F}$ is projective if and only if it is injective.
\end{proposition}
The proof can be found in \cite[Proposition 2.2.2]{BKN3-2009}.
Now we state several results on syzygies.

\begin{proposition} \label{P:syzygy-operation}
Let $M$ and $N$ be modules in $\mc{F}$, and let $m$, $n \in \Z$.  Then
\begin{enumerate}[(a)]
	\item $\Omega^0(M) \cong \Omega\inv ( \Omega^1(M))
		\cong \Omega^1( \Omega\inv(M))$;
	\item $\Omega^n(\Omega^m(M)) \cong \Omega^{n+m}(M)$ for any
		$n, m \in \Z$;
	\item \label{P:dualsyzygy} $(\Omega^n(M))^* \cong \Omega^{-n}(M^*)$ for
		any $n \in \Z$; 
	\item \label{L:tensorproj} if $P$ is a projective module in $\mc{F}$,
		 then $P \otimes N$ is also projective in $\mc{F}$;
	\item \label{P:syzygy tensor N} $\Omega^n(M) \otimes N
		\cong \Omega^n(M \otimes N) \oplus P$
		for some projective module in $\mc{F}$, $P$;
	\item \label{C: syzygy tensor syzygy} $\Omega^m(M) \otimes \Omega^n(N)
		\cong \Omega^{m+n}(M \otimes N) \oplus P$
		for some projective module in $\mc{F}$, $P$;
	\item \label{P:sumsyzygy} 
		$\Omega^n(M) \oplus \Omega^n(N) \cong \Omega^n(M \oplus N)$.
\end{enumerate}
\end{proposition}

\begin{proof}
(\ref{L:tensorproj})
Let $S$ be a module in $\mc{F}$ and $P$ and $N$ as above.
By the tensor identity stated in
\cite[Lemma 2.3.1]{BKN1-2006}, we have
$$
\Ext^n_{\mathcal{F}} (S, P \otimes N)
	\cong \Ext^n_{\mathcal{F}}(S \otimes N^*, P) = 0
$$
for $n > 0$ since $P$ is projective.  Thus, $P \otimes N$
is also projective in $\mc{F}$.

The other proofs are omitted as
they are very similar to the case of modules for group rings of finite
groups found in \cite{Benson1-1998}.
\end{proof}

We now have enough tools to prove the following.

\begin{proposition} \label{T:Syzygy}
If a module $M \in \mc{F}$ is endotrivial, then so is $\Omega^n(M)$ for
any $n \in \Z$.
\end{proposition}
\begin{proof}
By assumption, $M \otimes M^* \cong k_{ev} \oplus P$ for
some projective module $P$.  Applying Proposition
\ref{P:syzygy-operation} parts (\ref{P:dualsyzygy}) and
(\ref{P:syzygy tensor N}) yields
\begin{gather*}
\Omega^n(M) \otimes (\Omega^n(M))^* 
	 \cong \Omega^n(M) \otimes \Omega^{-n}(M^*) 
	 \cong \Omega^0(M \otimes M^*) \oplus P' 
	 \cong k_{ev} \oplus P'.
\end{gather*}
\end{proof}

Given a fixed Lie superalgebra, $\mf{g}$, we can consider
the set of endotrivial modules in $\mc{F}$.  We define 
$$
T(\mf{g}) := \left\{ [M] \in \op{Stmod}(\mc{F}) \st M \otimes M^* \cong
k_{ev} \oplus P_M \text{ where $P_M$ is projective in $\mc{F}$} \right\}.
$$
\begin{proposition}
Let $\mf{g}$ be a Lie superalgebra.  Then $T(\mf{g})$ forms
an abelian group under the operation $[M] + [N] = [M \otimes N]$.
\end{proposition}
\begin{proof}
Since tensoring with a projective module yields another projective module,
if $[M]$, $[N] \in T(\mf{g})$, then
\begin{gather*}
(M \otimes N) \otimes (M \otimes N)^* 
	= (k_{ev} \oplus P_M) \otimes (k_{ev} \oplus P_N) 
	= k_{ev} \oplus P_{M \otimes N}
\end{gather*}
and
so $[M \otimes N] \in T(\mf{g})$ as well, and the set is closed
under the operation $+$.  This operation is associative
by the associativity of the tensor product and commutative by the
canonical isomorphism $M \otimes N \cong N \otimes M$.

The isomorphism class of the trivial module $[k]$ is
the identity element and, since $M$ is endotrivial,
$[M]\inv = [M^*]$ which is also necessarily in $T(\mf{g})$ since $M$
is finite dimensional.
\end{proof}

The following lemma simplifies computations involving both syzygies
and the parity change functor and will be useful throughout this work.
\begin{lemma} \label{L: Pi is Omega invariant}
Let $k$ be either the trivial supermodule, $k_{ev}$,
or $\Pi(k_{ev}) = k_{od}$ in $\mc{F}$, then
$$
\Pi(\Omega^n(k)) = \Omega^n(\Pi(k))
$$
for all $n \in \Z$.
\end{lemma}
\begin{proof}
In the case where $n=0$, the claim is trivial.

The parity change functor, $\Pi$, can be realized by the following.
Let $M$ be a $\mf{g}$-supermodule, then
$$
\Pi(M) \cong M \otimes k_{od}
$$
and if $N$ is another $\mf{g}$-supermodule and
$\phi: M \rightarrow N$ is a $\mf{g}$-invariant map, then
\begin{gather*}
\Pi(\phi) : M \otimes k_{od} \rightarrow N \otimes k_{od} \\
m \otimes c \mapsto \phi(m) \otimes c
\end{gather*}
defines the functor $\Pi$.
Let
$$
\begin{tikzpicture}[start chain] {
	\node[on chain] {$0$};
	\node[on chain] {$\Omega^1(k) $} ;
	\node[on chain] {$P$};
	\node[on chain] {$k$};
	\node[on chain] {$0$}; }
\end{tikzpicture}
$$
be the exact sequence defining $\Omega^1(k)$.  Then
$\Pi(P)$ is the projective cover of $\Pi(k)$ and since the tensor product
is over $k$, the following sequence
$$
\begin{tikzpicture}[start chain] {
	\node[on chain] {$0$};
	\node[on chain] {$\Pi(\Omega^1(k)) $} ;
	\node[on chain] {$\Pi(P) $};
	\node[on chain] {$\Pi(k) $};
	\node[on chain] {$0$}; }
\end{tikzpicture}
$$
is exact.
Thus, $\Pi(\Omega^1(k)) = \Omega^1(\Pi(k))$ as desired.
We can easily dualize this argument to see that
$\Pi(\Omega\inv(k)) = \Omega\inv(\Pi(k))$ and the proof is completed
by an induction argument.
\end{proof}

\section{Computing $T(\mf{g})$ for Rank 1 Detecting Subalgebras}

Determining $T(\mf{g})$ for
different Lie superalgebras will be the main
goal for the next three sections.  In particular, we will be focusing
on the case where $\mf{g}$ is a detecting subalgebra, whose definition
depends on the Lie superalgebras $\mf{q}(1)$ and $\mf{sl}(1|1)$.

\subsection{Detecting Subalgebras} \label{SS: detecting subalgebras}

Recall the definitions of $\mf{q}(n) \subseteq \mf{sl}(n|n) \subseteq \mf{gl}(n|n)$.
The Lie superalgebra \define{$\mf{q}(n)$} consists of
$2n \times 2n$ matrices of the
form
\[ \left( \begin{array}{c|c}
		A & B \\ \hline
		B & A
		\end{array} \right)		\]
where $A$ and $B$ are $n \times n$ matrices over $k$.
The Lie superalgebra $\mf{q}(1)$ is of primary interest and has a basis of
\[ t = \left( \begin{array}{cc}
		1 & 0 \\
		0 & 1
		\end{array} \right)	\quad	
		e =			
		\left( \begin{array}{cc}
		0 & 1 \\
		1 & 0
		\end{array} \right).\] 
Note that $t$ spans $\ev{\mf{q}(1)}$ and $e$ spans $\od{\mf{q}(1)} $.
The brackets are easily computed using Equation \ref{E:bracket},
$$
[t,t] = tt - tt = 0, \quad [t,e] = te - et = 0, \quad [e,e] = ee + ee = 2t.
$$

The Lie superalgebra $\mf{sl}(m|n) \subseteq \mf{gl}(m|n)$ consists of
$(m + n) \times (m + n)$  matrices of the form
\[ \left( \begin{array}{c|c}
		A & B \\ \hline
		C & D
		\end{array} \right)		\]
where $A$ and $D$ are $m \times m$ and $n \times n$ matrices respectively,
which satisfy the condition $\operatorname{tr}(A) - \operatorname{tr}(D)=0$.

For this work, $\sL$ is of particular importance, so note that $\sL$ consists of
$2 \times 2$ matrices and has a basis of
\[ t = \left( \begin{array}{cc}
		1 & 0 \\
		0 & 1
		\end{array} \right)	\quad	
		x =			
		\left( \begin{array}{cc}
		0 & 1 \\
		0 & 0
		\end{array} \right) \quad
		y = \left( \begin{array}{cc}
				0 & 0 \\
				1 & 0
				\end{array} \right).\] 
A direct computation shows that
$
[x,y] = xy + yx = t
$
and that all other brackets in $\sL$ are 0.

We may now define
the detecting subalgebras, $\mf{e}$ and $\mf{f}$ as introduced in
\cite[Section 4]{BKN1-2006}.
Let $\mf{e}_m:= \mf{q}(1) \times \mf{q}(1) \times \dots \times \mf{q}(1)$
with $m$ products of $\mf{q}(1)$, and
$\mf{f}_n:= \mf{sl}(1|1) \times \mf{sl}(1|1) \times \dots \times
\mf{sl}(1|1)$ with $n$ products of $\mf{sl}(1|1)$.

Let $\mf{a}$ denote an arbitrary detecting subalgebra
(either $\mf{e}$ or $\mf{f}$).
We define the rank of a detecting subalgebra $\mf{a}$ to be $\dim(\od{\mf{a}})$,
the dimension of the odd degree.
The rank of $\mf{e}_m$ is $m$ and the rank of
$\mf{f}_n$ is $2n$ and, in general, a detecting subalgebra of rank $r$
is denoted $\mf{a}_r$.

With these definitions established, we now turn to the question of
classifying the endotrivial modules for rank 1 detecting subalgebras.

\subsection{Endotrivial Modules for the $\mf{e}_1$ Detecting Subalgebra}
\label{SS: e1 classification}
It is possible to classify $T(\mf{e}_1)$ by considering
the classification of all indecomposable $\mf{q}(1)$-modules found in
\cite[Section 5.2]{BKN1-2006}.

We know that $k \times \Z_2$ parameterizes
the simple $\mf{q}(1)$-supermodules and
the set $\{ L(\lambda), \Pi(L(\lambda)) \ | \ \lambda \in k \}$
is a complete set of simple
supermodules.
For $\lambda \neq 0$, $L(\lambda)$ and $\Pi(L(\lambda))$ are two
dimensional and projective.
Thus, the only simple modules that are not projective are $L(0)$, the
trivial module, and $\Pi(L(0))$.  The
projective cover of $L(0)$ is obtained as
$$
P(0)=U(\mf{q}(1)) \otimes_{U(\ev{\mf{q}(1)})} L(0)|_{\ev{\mf{q}(1)}}.
$$
Since $U(\mf{q}(1))$ has a basis (according to the PBW theorem) of
$$
\left\{ e^r t^s \st r \in \{0,1\}, \text{ } s \in \Z_{\geq 0} \right\},
$$  we see that $P(0)$ has a
basis of $\{1 \otimes 1, e \otimes 1 \}$ and there is
a one dimensional submodule spanned by $e \otimes 1$
which is isomorphic to $k_{od}$, the trivial module under the parity
change functor.  The space spanned by $1 \otimes 1$ is not closed under the action
of $U(\mf{q}(1))$ since $e(1 \otimes 1) = e \otimes 1$.
It is now clear that the structure of $P(0)$ is
\begin{equation*}
%\begin{center}
\begin{tikzpicture}[description/.style={fill=white,inner sep=2pt},baseline=(current  bounding  box.center)]

\matrix (m) [matrix of math nodes, row sep=3em,
column sep=2.5em, text height=1.5ex, text depth=0.25ex]
{ k_{\text{ev}}\\
  k_{\text{od}} \\ };

\path[-,font=\scriptsize]
	(m-1-1) edge node[auto] {$ e $} (m-2-1) ;
\end{tikzpicture}
%\end{center}
\end{equation*}
and the reader may check that $P(\Pi(L(0))) \cong\Pi(P(0))$.  Thus, directly
computing all indecomposable modules shows that the only indecomposables
which are not projective are $k_{ev}$ and $k_{od}$ which are clearly
endotrivial modules.

The group $T(\mf{e}_1)$ can be described in terms of the syzygies.
The kernel of the projection map from $P(0)$ to $L(0)$ is
$\Omega^1(k_{ev}) = k_{\text{od}}$.  In order to compute $\Omega^2(k_{ev})$,
consider the projective cover of ${k_\text{od}}$, which is $\Pi(P(0))$.
The kernel of the projection map is again $k_{\text{ev}}$, the trivial
module.  The situation is the same for $\Pi(L(0))$ only with the parity
change functor applied.

Now we have the following complete list of indecomposable
endotrivial modules,
\begin{align*}
\Omega^n(k_{ev}) =
\begin{cases}
k_{\text{ev}} & \text{if $n$ is even}\\
k_{\text{od}} & \text{if $n$ is odd}
\end{cases} & \qquad
\Omega^n(k_{od}) =
\begin{cases}
k_{\text{od}} & \text{if $n$ is even}\\
k_{\text{ev}} & \text{if $n$ is odd}
\end{cases}
\end{align*}
and an application of Proposition \ref{P:syzygy-operation} 
(\ref{C: syzygy tensor syzygy}) proves the
following proposition.
\begin{proposition} \label{P: T(e1)}
Let $\mf{e}_1$ be the rank one detecting subalgebra of type $\mf{e}$.  Then
$T(\mf{e}_1) \cong \Z_2$.
\end{proposition}

\subsection{Endotrivial Modules for the $\mf{f}_1$ Detecting Subalgebra}
\label{SS: f1 classification}
Now we consider the Lie superalgebra $\mf{sl}(1|1)$.
By the PBW theorem, a basis of $U(\sL)$ is given by
\begin{equation} \label{E: U(sl(1|1)) basis}
\left\{ x^{r_1}y^{r_2}t^{s}  \st r_i \in \{0,1 \}
\text{ and } s \in \Z_{\geq 0}  \right\}. %\tag{$*$}
\end{equation}
Not all endotrivial $U(\sL)$-supermodules will be classified yet since
this is a rank 2 detecting subalgebra.
First consider modules over the Lie superalgebra generated by one
element of $\od{\sL}$.

Note that, since $[x,x] = [y,y] = 0$, an
$\langle x \rangle$-supermodule or a
$\langle y \rangle$-supermodule will also fall under the classification
given in \cite[Section 5.2]{BKN1-2006}.  For modules of this type,
there are only four isomorphism classes of indecomposable
modules, $k_{ev}$, $k_{od}$,
$U(\langle x \rangle)$, and $\Pi(U(\langle x \rangle))$.
It can be seen
by direct computation that $U(\langle x \rangle)$ is the projective
cover of $k_{ev}$ (and the kernel of the projection map is $k_{od}$) and
$\Pi(U(\langle x \rangle))$ is the projective
cover of $k_{od}$ (and the kernel of the projection map is $k_{ev}$).

Alternatively, let $z = ax + by$ where $a$, $b \in k\setminus\{0\}$.
then $U(\langle z \rangle) \cong U(\mf{q}(1))$.
Thus, we have the same result and proof as in Proposition \ref{P: T(e1)}.

\begin{proposition} \label{P: T(sl(1|1)) rank 1 subalgebra}
Let $\mf{f}_1|_{\langle z \rangle}$ be a rank 1 subalgebra
of $\sL$ generated by $z$, an element of
$\od{\sL}$.  Then $T(\mf{f}_1|_{\langle z \rangle}) \cong \Z_2$.
\end{proposition}

\section{Computing $T(\mf{g})$ for Rank 2 Detecting Subalgebras}

The main result of this section is the classification of $T(\mf{a}_2)$,
stated in Theorem \ref{T: T(g) for rank 2}.
Given this goal, the Lie superalgebras of interest in this section are
$\mf{q}(1) \times \mf{q}(1)$, denoted $\mf{e}_2$, and $\sL$, denoted $\mf{f}_2$.
The classification of $T(\mf{a}_2)$ is more complex than the rank one
case and will require some general information about representations
of the detecting subalgebras to prove the main theorem of this section.

\subsection{Rank r Detecting Subalgebras} \label{SS: rank r detecting}

Since the this section requires considering arbitrary detecting
subalgebras, consider the following.  Rank $r$ detecting
subalgebras are defined to be subalgebras isomorphic to either
$\mf{e}_r \cong \mf{q}(1) \times \dots \times \mf{q}(1) \subseteq \mf{gl}(n|n)$
or if $r$ is even, $\mf{f}_{r/2} \cong \sL \times \dots \times \sL \subseteq \mf{gl}(n|n)$
where there are $r$ and $r/2$ copies of $\mf{q}(1)$ and $\sL$ respectively.

Recall, $\mf{e}_r$ has a basis of
$$
\{ e_1, \cdots,  e_r, t_1, \cdots, t_r \}
$$
and there are matrix representations of $t_i$ and $e_i$ which
are $2r \times 2r$
matrices with blocks of size $r \times r$.  Let $d_i$ be the $r \times r$
matrix with a 1 in the $i$th diagonal entry and $0$ in all other entries.
Then
\[ t_i = \left( \begin{array}{c|c}
		d_i & 0   \\ \hline
		0 & d_i  \\
		\end{array} \right)	\quad	
	e_i =			
	\left( \begin{array}{c|c}
			0 & d_i   \\ \hline
			d_i & 0  \\
			\end{array} \right)\]
is a representation of $\mf{e}_r$.  The only nontrivial
bracket operations on $\mf{e}_r$ are $[e_i, e_i] = 2t_i$, thus
all generating elements in $U(\mf{e}_r)$ commute except for $e_i$ and $e_j$
anti-commute when $i \neq j$ and
by the PBW theorem,
$$
\left\{ e_1^{k_1}\cdots e_r^{k_r}t_1^{l_1} \cdots t_r^{l_r} \st
k_i \in \{0, 1 \} \text{ and } l_i \in \Z_{\geq 0}    \right\}
$$
is a basis for $U(\mf{e}_r)$.

For $\mf{f}_{r/2}$, the set
$$
\left\{ x_1, \cdots, x_{r/2},
y_1, \cdots, y_{r/2}, t_1, \cdots, t_{r/2}   \right\}
$$
forms a basis.
The matrix representation of each element are also
$2r \times 2r$ matrices with blocks of size $r \times r$.
If $d_i$ is as above then
\[ t_i = \left( \begin{array}{c|c}
		d_i & 0   \\ \hline
		0 & d_i  \\
		\end{array} \right)	\quad	
	x_i =			
	\left( \begin{array}{c|c}
			0 & d_i   \\ \hline
			0 & 0  \\
			\end{array} \right) \quad	
	y_i =			
	\left( \begin{array}{c|c}
			0 & 0   \\ \hline
			d_i & 0  \\
			\end{array} \right)\]
gives a realization of $\mf{f}_{r/2}$.  The only nontrivial brackets
are $[x_i,y_i]=t_i$.  So, in $U(\mf{f}_{r/2})$, $x_i \otimes y_j = 
-y_j \otimes x_i$ when $i \neq j$ and $t_i$ commutes with $x_j$ and
$y_k$ for any $i$, $j$, and $k$.  Observe that
$x_i \otimes y_i = -y_i \otimes x_i + t_i$ for each $i$.
Finally, the PBW theorem shows that
$$
\left\{ x_1^{k_1}\cdots x_{r/2}^{k_{r/2}}
y_1^{l_1}\cdots y_{r/2}^{l_{r/2}}t_1^{m_1} \cdots t_{r/2}^{m_{r/2}} \st
k_i, l_i \in \{0, 1 \} \text{ and } m_i \in \Z_{\geq 0}    \right\}
$$
is a basis for $U(\mf{f}_{r/2})$.

Note that if $\mf{g}$ and $\mf{g}'$ are two Lie superalgebras,
and $\sigma : \mf{g} \rightarrow U(\mf{g})$ and
$\sigma' : \mf{g}' \rightarrow U(\mf{g}')$ are the canonical inclusions,
then
\cite{Scheunert-1979} gives an isomorphism between
$U(\mf{g} \times \mf{g}')$ and the graded tensor product
$U(\mf{g}) \ov{\otimes} U(\mf{g}')$
where the mapping 
\begin{gather*}
\tau:\mf{g} \times \mf{g}' \rightarrow U(\mf{g}) \ov{\otimes} U(\mf{g}') \\ \nonumber 
\tau(g, g') = \sigma(g) \otimes 1 + 1 \otimes \sigma'(g')
\end{gather*}
corresponds to the canonical inclusion of $\mf{g} \times \mf{g}'$ 
into $U(\mf{g} \times \mf{g}')$.  The corresponding construction for
modules is the outer tensor product.
If $M$ is
a $U(\mf{g})$-module and $N$ is a $U(\mf{g}')$-module, then the
outer tensor product of $M$ and $N$ is denoted $M \boxtimes N$.  This
is a $U(\mf{g}) \ov{\otimes} U(\mf{g}')$-module where
the action is given by
$$
(x \otimes y )(v \boxtimes w) =\sgndeg{y}{v} x(v) \boxtimes y(w)
$$
for $x \otimes y \in U(\mf{g}) \ov{\otimes} U(\mf{g}')$.

This correspondence is relevant to the work here because
it implies that $U(\mf{e}_r)
\cong U(\mf{e}_1) \ov{\otimes} \dots \ov{\otimes} U(\mf{e}_1)$ and 
$U(\mf{f}_{r/2}) \cong U(\sL)  \ov{\otimes} \dots \ov{\otimes} U(\sL)$
with $r$ and $r/2$ graded tensor products respectively.
It will be useful to think of the universal enveloping algebra and
corresponding representations in
both contexts.

\subsection{Representations of Detecting Subalgebras} \label{SS: reps of detecting sub}

Because $\ev{\mf{q}(1)}$ and $\ev{\sL}$ consist only of toral elements,
for an arbitrary detecting subalgebra $\mf{a}$,
the even component $\ev{\mf{a}}$ consists of only toral elements as well.
Thus, any $\mf{a}$-module will decompose into a direct sum of weight
spaces over the torus $\ev{\mf{a}}$.  Furthermore, this torus commutes
with all of $\mf{a}$, as can be observed by the bracket computations
given above,
and so the weight space decomposition actually yields a decomposition
as $\mf{a}$-modules.

This decomposition is consistent with the standard notion of a block
decomposition for modules and a case of particular importance
is that of the principal block, i.e., the block which contains the
trivial module.  All elements of $\mf{a}$ act by zero on the trivial
$\mf{a}$-module, and so in particular, all toral elements do as well.
This defines
the weight associated with this block to be the zero weight and
$\ev{\mf{a}}$ acts by zero on any module in the principal block.

This structure has implications for the projective and simple modules
in each block.  As noted, $\ev{\mf{a}}$ is toral, and hence
the only simple modules are one dimensional with weight $\lambda$.
Then for a simple $\mf{a}$-module $S$ of weight $\lambda = (\lambda_1,
\dots, \lambda_r)$, the restriction
$S|_{\ev{\mf{a}}} \cong
\oplus T_i$ where each $T_i$ is a simple, hence one dimensional,
$\ev{\mf{a}}$-module of weight $\lambda_i$.  Thus, $T_i \cong k_{\lambda_i}$
where $k_{\lambda_i}$ denotes a one dimensional module,
concentrated in either even or odd degree, with basis $\{ v \}$ and
the action of $\ev{\mf{a}}$ is given by $t_i.v = \lambda_i v$.
Then by Frobenius reciprocity,
$$
\Hom_{U(\mf{a})}(\ind{a} k_{\lambda_i}, S) \cong
\Hom_{U(\ev{\mf{a}})} (k_{\lambda_i}, S|_{\ev{\mf{a}}}) \neq 0
$$
and since $S$ is simple, the nonzero homomorphism is also surjective.
As noted in \cite{Kumar-2002}, these induced modules are projective
in $\mc{F}$ and the above result shows that
the projective cover of any simple $\mf{a}$-module can be found
as a direct summand of an induced one dimensional module.

When $\mf{a} \cong \sL$, these modules are small enough to
be explicitly described and are important for the next proposition.
By the considering the basis in Equation \ref{E: U(sl(1|1)) basis}, these
induced modules will be 4 dimensional.  Furthermore, when
$\lambda = 0$ and concentrated in the even degree,
the induced module, denoted $P(0)$, is indecomposable with simple socle
and simple head, and hence is the projective cover of the trivial module
$k_{ev}$ and $\Pi(P(0))$ is the projective cover of
$\Pi(k_{ev}) = k_{od}$.

When $\lambda \neq 0$, a direct computation shows that
the induced module splits as a direct sum of
two simple $\sL$-modules each of which are two dimensional, with
basis $\{ v_1, v_2 \}$ and action
\begin{center} 
\begin{tabular}{lcl}
$x.v_1 = v_2$	& and	& $x.v_1 = 0$	\\ 
$y.v_1 = 0$		&		& $y.v_2 = v_1$.	\\
\end{tabular}
\end{center}
For one of the summands $v_1$ is even and $v_2$ is odd, and for the
other $v_2$ is even and $v_1$ is odd.  Thus the simple $\sL$-modules
outside of the principal block are two dimensional and projective.
Moreover, in the terminology of \cite{BK-2002}, this computation shows
that these modules are absolutely irreducible.  This will be relevant
in the following proposition.

Another
useful property of this block decomposition is that it yields information
about the dimensions of the modules in certain blocks.
\begin{proposition}
Let $\mf{a}_r$ be a rank $r$ detecting subalgebra.  Then any simple module
outside the principal block has even dimension.
\end{proposition}
\begin{proof}
Since $\ev{\mf{q}(1)} = \ev{\sL}$, let $\{ t_1, \dots, t_r \}$ and
$\{ t_1, \dots, t_{r/2} \}$
be the corresponding bases for $\ev{(\mf{a}_r)}$ such that $t_i$ is a basis for
the even part of
the $i$th component of either $\mf{q}(1)$ or $\sL$.  Since the module is
outside of the principal block, the associated weight
$\lambda = (\lambda_1, \dots, \lambda_r)$ or
$\lambda = (\lambda_1, \dots, \lambda_{r/2})$ must have some
$\lambda_i \neq 0$.

According to a theorem of Brundan from \cite[Section 4]{Brundan-2002},
the simple $\mf{e}_r = \mf{q}(1) \times \dots \times \mf{q}(1)$
($r$ products)
modules, denoted $\mf{u}(\lambda)$ for $\lambda \in \Z^n$,
have characters given by
\begin{equation*}
\operatorname{ch} \mf{u}(\lambda) = 2^{\lfloor (h(\lambda)+1)/2 \rfloor} x^{\lambda}
\end{equation*}
where $h(\lambda)$ denotes the number of $t_i$ which do not act by 0.
Thus all modules are even
dimensional except in the case when $h(\lambda) = 0$, i.e. all
simple modules outside of the principal block are even dimensional.

For the case of $\mf{f}_{r/2} $, % = \sL \times \dots \times \sL$,
by \cite[Section 2.1, Proposition 2]{Scheunert-1979},
$U(\mf{f}_{r/2}) \cong U(\sL)  \ov{\otimes} \dots \ov{\otimes} U(\sL)$.
By \cite[Lemma 2.9]{BK-2002}, we can construct any irreducible
$\mf{f}_{r/2}$-module
as outer tensor products of irreducible $\sL$-modules.  Since all simple
$\sL$-modules are absolutely irreducible, the outer tensor product of
such modules is also absolutely irreducible.  We saw that a simple
$\sL$-module has dimension one if the weight is 0, and dimension two
otherwise.  Thus the dimension of
a simple $\mf{f}_{r/2}$-module
is $2^{h(\lambda)}$, and so all simple modules outside
of the principal block are even dimensional.
\end{proof}

Now that something is known about the dimensions of the simple
$\mf{a}$-modules, we consider the dimensions of the projective
$\mf{a}$-modules.  By the previous proposition, any projective
module outside of the simple block will be even dimensional as well.
Furthermore, in the previous sections we have shown
that the only simple modules in the principal block are
$k_{ev}$ and $k_{od}$, and for $\mf{e}_1$ and $\mf{f}_1$, the
direct computations have shown the projective covers of these are
indecomposable modules of even dimension.

By the rank variety theory of \cite[Section 6]{BKN1-2006}, restriction
of any projective $\mf{e}_r$ or $\mf{f}_r$ module, must be projective
when restricted to $\mf{e}_1$ or $\mf{f}_1$, respectively.  Thus,
any projective $\mf{a}$-module is a direct sum of modules which are
each even dimensional, and thus even dimensional as well.

Given these results, the following lemma will
greatly restrict our search for endotrivial modules.
\begin{lemma} \label{L: endo in principal}
Let $M$ be an indecomposable $\mf{a}$-supermodule.
If $M$ is an endotrivial
supermodule, then $M$ must be in the principal block, i.e. all of
the even elements of $\mf{a}$ must act on $M$ by 0.
\end{lemma}
\begin{proof}
Since $M$ is endotrivial, $M \otimes M^* \cong k_{ev} \oplus P$
for some projective module $P$.  Since $\dim P = 2m$ for $m \in \N$
by the previous observations, $\dim M \otimes M^* = \dim M^2 \equiv 1 \pmod{2}$.  Since all modules outside of the principal block are
even dimensional, $M$ must be in the principal block.
\end{proof}
This simplifies the search for endotrivial modules and we
also can conclude that the only simple endotrivial modules are
$k_{ev}$ and $k_{od}$.  Now we wish to show that the only endotrivials
are $\{ \Omega^n(k_{ev}), \Pi(\Omega^n(k_{ev})) | n \in \Z \}$.

\begin{note}
Since endotrivial $\mf{a}$-modules are restricted to the principal block,
the even elements act via the zero map on any module.  With this in mind,
it is convenient to think of endotrivial $\mf{a}$-supermodules in a
different way. Since $t_i$ acts trivially for all $i$, considering
$\mf{a}$-modules as a $\od{\mf{a}}$-modules with trivial bracket
yields an equivalent
representation. The representations of these superalgebras are equivalent
and the notation for this simplification is
$V(\mf{a}) := \Lambda(\od{\mf{a}})$.
\end{note}

With these reductions,
in general, endotrivial $\mf{a}_r$-modules are simply endotrivial modules
for an abelian Lie superalgebra of dimension $r$ concentrated in degree
$\ov{1}$.  For simplicity, denote a basis for $\od{(\mf{a}_r)}$ by
$\{a_1, \dots, a_r \}$.  Then it is clear that
$V(\mf{a}_r) = \langle 1, a_1, \dots, a_r \rangle$ generated as
an algebra.

\subsection{Computing $T(\mf{a}_2)$} \label{SS: T(a2)}
Given the above simplification, a basis for $V(\mf{e}_2)$
is given by
$$
\left\{ e_1^{r_1}e_2^{r_2} \st r_i \in \{0,1 \}  \right\}.
$$
If we consider the left regular representation of $V(\mf{e}_2)$
in itself with this basis, we have a 4 dimensional module with the structure
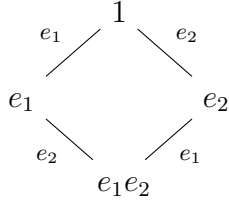
\begin{equation*}
%\begin{center}
\begin{tikzpicture}[description/.style={fill=white,inner sep=2pt},baseline=(current  bounding  box.center)]

\matrix (m) [matrix of math nodes, row sep=1.3em,
column sep=1em, text height=1.5ex, text depth=0.25ex]
{ & 1 & \\
  e_1 & & e_2 \\
   & \text{ } e_1  e_2 & \\ };

\path[-,font=\scriptsize]
	(m-1-2) edge node[auto, swap] {$ e_1 $} (m-2-1)
			edge node[auto] {$ e_2 $} (m-2-3)
	(m-3-2) edge node[auto] {$ e_2 $} (m-2-1)
			edge node[auto, swap] {$ e_1 $} (m-2-3);
\end{tikzpicture}
%\end{center}
\end{equation*}
which is isomorphic to
the projective cover of $k$ thought of as an $\mf{e}_2$-module.

In the $\sL$ case, now
that the search has been restricted to the principal block,
$V(\sL)$ has the same structure.  $V(\sL)$ has a basis
given by
$$
\left\{ x^{r_1}y^{r_2}  \st r_i \in \{0,1 \}  \right\}
$$
and $V(\sL)$ has the structure
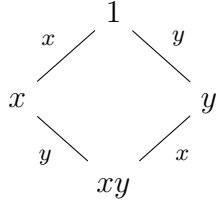
\begin{equation*}
%\begin{center}
\begin{tikzpicture}[description/.style={fill=white,inner sep=2pt},baseline=(current  bounding  box.center)]

\matrix (m) [matrix of math nodes, row sep=1.3em,
column sep=1.6em, text height=1.5ex, text depth=0.25ex]
{ & 1 & \\
  x & & y \\
   & \text{} xy \text{} & \\ };

\path[-,font=\scriptsize]
	(m-1-2) edge node[auto, swap] {$ x $} (m-2-1)
			edge node[auto] {$ y $} (m-2-3)
	(m-3-2) edge node[auto] {$ y $} (m-2-1)
			edge node[auto, swap] {$ x $} (m-2-3);
\end{tikzpicture}
%\end{center}
\end{equation*}
which is isomorphic to that of $V(\mf{e}_2)$ and is the projective
cover of $k$ as an $\sL$-module.

For the rank 2 case, the
algebra $V(\mf{a}_2)$ is similar to the group
algebra $k(\Z_2 \times \Z_2)$ and endotrivials
in the superalgebra case will be classified using a similar approach to
Carlson in \cite{Carlson-1980}.  First we give
analogous definitions and constructions to those
in Carlson's paper. 

Let $M$ be a $\mf{g}$-supermodule.  The rank of $M$, denoted $\Rk (M)$, is
defined by $\Rk(M) = \dim_k (M/\operatorname{Rad}(M))$.
Any element of $M$ which is not in $\op{Rad}(M)$
will be referred to as a generator of $M$.
The socle of $M$
has the standard definition (largest semi-simple submodule) and can be
identified in the case of the principal block by
$
\operatorname{Soc} M = \{m \in M | u.m=0
\text{ for all } u \in \op{Rad}( V(\mf{a}_r)) \}.
$
If $\mf{h}$ is a subalgebra of $\mf{a}_r$
with a basis of $\od{\mf{h}}$ given by $\{h_1, \dots, h_s \}$,
then
$\tilde{\mf{h}}:= \bigotimes_{i=1}^s h_i$ is a useful
element of $V(\mf{h})$.  This is because in $M|_{V(\mf{h})}$,
$\tilde{\mf{h}}.M \subseteq \operatorname{Soc} M|_{\V{h}}$.

Now we prove a lemma in the same way as \cite{Carlson-1980}.
\begin{lemma} \label{L: endo is indec plus proj}
Let $M$ be an endotrivial $\mf{a}_r$-supermodule for any
$r \in \N$.  Then
$$
\dim \Ext_{V(\mf{a}_r)}^1 (M, \Omega^1(M)) = 1
$$ and $M$ is the direct
sum of an indecomposable endotrivial module and a projective module.
\end{lemma}
\begin{proof}
By definition, $\Hom_k (M, M) \cong k_{ev} \oplus P$ for some projective module
$P$.  It is clear from the definitions, that
$\Hom_{V(\mf{a}_r)}(M,M) = \Soc (\Hom_k (M,M))$.  We have observed that
$\tilde{\mf{a}}.M \subseteq \Soc (M)$ and in the case of a projective
module, equality holds.  So then
$$
\tilde{\mf{a}}.\Hom_k (M,M) = \tilde{\mf{a}}.(k \oplus P) = \Soc (P).
$$
Since $\Soc (\Hom_k (M,M)) = k \oplus \Soc(P)$, we can see that
$\tilde{\mf{a}}.\Hom_k (M,M)$ is a submodule of $\Hom_{V(\mf{a}_r)} (M,M)$
of codimension one.

Let $P'$ be the projective cover of $M$.  Apply $\Hom_k (M, -)$
and the long exact sequence in cohomology
to the short exact sequence defining $\sy{1}{M}$
to get the following commutative diagram
\begin{equation*}
%\begin{center}
\begin{tikzpicture}[description/.style={fill=white,inner sep=2pt},baseline=(current  bounding  box.center)]

\matrix (m) [matrix of math nodes, row sep=1.5em,
column sep=1.5em, text height=1.5ex, text depth=0.25ex]
{	0	&	\Hom_k (M, \Omega^1(M))	&	\Hom_k (M, P')
										&	\Hom_k(M,M) & 0 \\
	0	&	\Hom_{V(\mf{a}_r)}(M, \Omega^1(M))
		&	\Hom_{V(\mf{a}_r)} (M, P')
		&	\Hom_{V(\mf{a}_r)} (M,M) \\
	\text{ } &	\Ext^1_{V(\mf{a}_r)} (M, \Omega^1(M)) & 0 \qquad\qquad\\};

\path[->,font=\scriptsize]
	(m-1-1) edge (m-1-2)
	(m-1-2) edge (m-1-3)
			edge (m-2-2)
	(m-1-3) edge (m-1-4)
			edge (m-2-3)
	(m-1-4) edge (m-1-5)
			edge (m-2-4)
	(m-2-1) edge (m-2-2)
	(m-2-2) edge (m-2-3)
	(m-2-3) edge node[auto] {$ \psi^* $} (m-2-4)
	(m-3-1) edge (m-3-2)
	(m-3-2) edge (m-3-3); 
\end{tikzpicture}
%\end{center}
\end{equation*}
where the vertical maps are multiplication by $\tilde{\mf{a}}$.
Since the diagram commutes and the map into $\Hom_k (M,M)$ is surjective,
the image of $\psi^*$ contains $\tilde{\mf{a}}.\Hom_k (M,M)$,
and we see that the dimension of
$\Ext^1_{V(\mf{a}_r)} (M, \Omega^1(M))$
is at most 1.  The dimension is nonzero since the extension
between a non-projective module and the first syzygy does not split.  Since
$\Ext_{V(\mf{a}_r)}^1$ splits over direct sums, the claim is established.
\end{proof}

\begin{lemma} \label{L: restriction commutes}
Let $M$ be a $V(\mf{a}_r)$-module and let $\mf{b}$ be a subalgebra
of $\mf{a}_r$.  Then
$$
\Omega^n_{\mf{a}_r}(M)\big|_{\mf{b}}
\cong \Omega^n_{\mf{b}}(M|_{\mf{b}}) \oplus P
$$
for all $n \in \Z$, where $P$ is a projective $V(\mf{b})$-module.
\end{lemma}
\begin{proof}
The case when $n=0$ is proven by considering the rank varieties
of $\mf{a}_r$ and $\mf{b}$. Let $M$ be as above and
$$
\begin{tikzpicture}[start chain] {
	\node[on chain] {$0$};
	\node[on chain] {$\Omega^1_{\mf{a}_r}(M)$} ;
	\node[on chain] {$P$};
	\node[on chain] {$M$};
	\node[on chain] {$0$}; }
\end{tikzpicture}
$$
be the short exact sequence of $V(\mf{a}_r)$-supermodules defining
$\Omega^1(M)$.  Then by the rank variety theory of
\cite[Section 6]{BKN1-2006}, the module $P|_{\mf{b}}$ is a projective 
$V(\mf{b})$-module (although perhaps not the projective cover
of $M|_\mf{b}$) and
$$
\begin{tikzpicture}[start chain] {
	\node[on chain] {$0$};
	\node[on chain] {$\Omega^1_{\mf{a}_r}(M)|_{\mf{b}}$} ;
	\node[on chain] {$P|_{\mf{b}}$};
	\node[on chain] {$M|_{\mf{b}}$};
	\node[on chain] {$0$}; }
\end{tikzpicture}
$$
is still exact.  Then by definition,
$\Omega^1_{\mf{a}_r}(M)|_{\mf{b}} \cong
\Omega^1_{\mf{b}}(M|_{\mf{b}}) \oplus P$.
This argument applies to $\Omega\inv(M)$ as well and so by induction,
$\Omega^n_{\mf{a}_r}(M)|_{\mf{b}} \cong \Omega^n_{\mf{b}}(M|_{\mf{b}})
\oplus P$
for all $n \in \Z$
\end{proof}
The previous lemma indicates that the syzygies of a module commute with
restriction up to a projective module (and truly commute in the stable
module category).  With this in mind, the subscripts on the syzygies
will be suppressed as it will be clear from context which syzygy to
consider.

In proving Theorem \ref{T: T(g) for rank 2},
we work with 2 conditions on a $V(\mf{a}_2)$-supermodule, $M$.  For
notation, recall that
$V(\mf{a}_2) = \gen{1, a_1, a_2}$.  The conditions are
\begin{enumerate}[(1)]
\item $\Rk (M) > \Rk (\Omega\inv (M))$;
\item for any nonzero $a = c_1 a_1 + c_2 a_2 \in \od{(\mf{a}_2)}$,
$M|_{\gen{a}} \cong \Omega^t(k|_{\gen{a}}) \oplus P$
where $t = 0$ or 1,
$P$ is a projective $\gen{a}$-supermodule, and $c_1, c_2 \in k$.
\end{enumerate}

Note that in condition (2), $U(\gen{a}) \cong V(\mf{a}_1)$ and so
the structure of such modules is detailed in Sections
\ref{SS: e1 classification} and \ref{SS: f1 classification}.

The technique is to show that some endotrivial modules have these
properties and use them to classify all endotrivial modules.

Let $M$ be an endotrivial $V(\mf{a}_2)$ module.  Because the complexity
of an $M$ is nonzero (see \cite[Corollary 2.71]{BKN3-2009}),
$\Omega^n(M)$ satisfies (1) for some $n \in \Z_{\geq 0}$,
so we can simply replace
$M$ with $\sy{n}{M}$ initially and proceed.  Since
the endotrivials for rank 1 supermodules have been classified,
and $M|_{\gen{a}}$ is
isomorphic to a  $V(\mf{a}_1)$-supermodule, $M$ satisfies (2)
as well.

The classification of $T(\mf{a}_2)$ is approached in the same way
as in \cite{Carlson-1980}, but
the techniques are altered to suit the rank 2 detecting subalgebra case.
We begin by establishing several supplementary results.

\begin{lemma} \label{L: 2 holds}
A $V(\mf{a}_2)$-supermodule $M$ satisfies (2) if
and only if $\Omega^n(M)$ satisfies (2) for all $n \in \Z$.
\end{lemma}
\begin{proof}
Since $M|_{\gen{a}}$ is endotrivial, then $\Omega^n(M)|_{\gen{a}}$
is as well by Lemma \ref{L: restriction commutes}, and so by the classification of the $V(\mf{a}_1)$ endotrivials, 
$\Omega^n(M)$ satisfies (2) for all $n$ if and only if $M$ satisfies (2).
\end{proof}

\begin{lemma} \label{L:gen and rank}
Let $M$ be a $V(\mf{a}_2)$-supermodule which has no projective
submodules and which satisfies condition (2) and let $b$ be some
nonzero element
of $\od{(\mf{a}_2)}$ not in the span of $a$.
Let $v$ be a generator for the $\Omega^t(k|_{\gen{a}})$ component of some
decomposition of $M|_{\gen{a}}$, then $M$ satisfies condition (1) if
and only if every such element $v$ satisfies 
$b.v \neq 0$.

\end{lemma}
\begin{proof}
First, note that since $M$ has no projective summands, $a.M \subseteq
\Ann_b(M)$ and $b.M \subseteq \Ann_a(M)$, i.e. that radical series of
$M$ has length 1.  Otherwise if $(a \otimes b).m \neq 0$ for some
$m \in M$, then this would generate the socle of a projective submodule,
and hence summand since $V(\mf{a}_2)$ is self injective.
Recall also that in the principal block $\Ann_x(M) = \Soc(M|_{\gen{x}})$.

Let $M$ be as above and assume that $b.v \neq 0$.
Since $M$ is endotrivial, it must have odd dimension, and by the 
decomposition given in (2) and the knowledge of projective 
$U(\gen{a})$-modules, if $\dim M = 2n + 1$, then
$\dim \Soc(M|_{\gen{a}}) = n+1$.

Since $M$ satisfies
(2) and the choice of $a$ was arbitrary, we also know that
$M|_{\gen{b}} \cong \Omega^t(k|_{\gen{b}}) \oplus P$ and since
$b.v \neq 0$, $v$ is not in the socle of $M|_{\gen{b}}$, which also
has dimension $n+1$.  Since, as noted before,
$a.M \subseteq \Ann_b(M)$ and $b.M \subseteq \Ann_a(M)$, but
$b.v \neq 0$, then $\Soc(M) = \Ann_a(M) \cap \Ann_b(M) = a.M = b.M$ which has
dimension $n$ and $\Rk(M) = n+1$.  Thus, if
$0 \rightarrow M\rightarrow I \rightarrow \Omega\inv(M) \rightarrow 0$
is the exact sequence defining $\Omega\inv(M)$, then since $M$
has no projective submodules, the following holds
\begin{equation} \label{E: rank relation}
\Rk (\Omega\inv(M)) = \Rk (I) = \dim (\Soc(M)),
\end{equation}
and so
$$
\Rk(M) = n+1 > n = \Rk (\Omega\inv(M))
$$
as was desired.

Now, assume that $\Rk(M) > \Rk (\Omega\inv(M))$ and let $a$ and $b$
be as above.  Since $M|_{\gen{a}} \cong \Omega^t(k|_{\gen{a}}) \oplus P$
and $M|_{\gen{b}} \cong \Omega^t(k|_{\gen{b}}) \oplus P'$ which both
have socles of dimension $n + 1$ and intersect in at least $n$ dimensions
since $a.M \subseteq \Ann_b(M)$ and $b.M \subseteq \Ann_a(M)$ and
at most in $n+1$.  However,
by Equation \ref{E: rank relation}, they cannot intersect in $n+1$
or else condition (1) would be violated.  Thus, if $v$ is 
a generator for the $\Omega^t(k|_{\gen{a}})$ component of $M|_{\gen{a}}$,
then it is not in the socle of $M|_{\gen{b}}$, and so $b.v \neq 0$.
\end{proof}

A new condition is introduced to encapsulate the previous lemma.

\begin{enumerate}[(1)]
\setcounter{enumi}{2}
\item Let $M$ satisfy condition (2).  If $v$ is a generator for
the $\Omega^t(k|_{\gen{a}})$ component of $M|_{\gen{a}}$, then
$v$ is a generator for $M$.
If $m$ is any generator for $M$ then either
$a.m \neq 0$ or $b.m \neq 0$ where $b$ is some nonzero element of
$\od{(\mf{a}_2)}$ not in the span of $a$.
\end{enumerate}

\begin{lemma} \label{L: 2 then 4 iff 1}
Let $M$ be a $V(\mf{a}_2)$-supermodule which has no projective
submodules and which satisfies condition (2), then $M$ satisfies (1)
if and only if $M$ satisfies (3).
\end{lemma}
\begin{proof}
First, assume $M$ satisfies (3).  Since $M$ satisfies (2), let
$v$ be some generator for the $\Omega^t(k|_{\gen{a}})$ component in some
decomposition of $M|_{\gen{a}}$.  Since $a.v = 0$,
$b.v \neq 0$ by (3)
and by Lemma \ref{L:gen and rank}, $M$ satisfies (1).

Next, assume that $M$ satisfies (1).  By Lemma \ref{L:gen and rank},
$b.v \neq 0$.  Assume that $v$ is not a generator
for $M$.  Then, by definition, $v$ is in $\op{Rad}(M)$ and so we can write
$v = a.m_1 + b.m_2$ for some $m_1, m_2 \in M$.  By
assumption
$$
b.v = b.(a.m_1 + b.m_2) = c\widetilde{\mf{a}_2}.m_1 \neq 0
$$
for some $c \in k$.  However,
then $m_1$ generates a projective submodule which is a contradiction,
so $v$ is also a generator of $M$.

The only thing left to show is that if $m$ is any generator
(i.e. not in the radical) for $M$,
either $a.m \neq 0$ or $b.m \neq 0$.  If $m$ is any such generator,
then $m$ must also be a generator for $M|_{\gen{a}}
\cong \Omega^t(k|_{\gen{a}}) \oplus P$.
The case when
$m$ generates $\Omega^t(k|_{\gen{a}})$ has been handled, so assume
that $m$ is a generator for $P$.  Since $P$ is a projective
$\gen{a}$-module, and $m$ is a generator, then $a.m \neq 0$.
\end{proof}

\begin{proposition} \label{P: 4 implies syzygy 4}
Let $M$ be a $V(\mf{a}_2)$-supermodule which has no projective submodules
satisfying (1), (2), and (3),
then $\Omega^n(M)$ does as well for all $n \geq 0$.
\end{proposition}
\begin{proof}
By Lemmas \ref{L: 2 holds} and \ref{L: 2 then 4 iff 1}, it is
sufficient to show that (3) holds for $\Omega^n(M)$ for all $n \geq 0$.
This proposition it trivial when $n=0$ since $\Omega^0(M) = M$ by
assumption.

Let $m$ be a generator for $\Omega^1(M)$, then $m$ is also a generator for
$\Omega^1(M)|_{\gen{a}} \cong \Omega^t(k|_{\gen{a}}) \oplus P$ in some
decomposition of $\Omega^1(M)|_{\gen{a}}$.  Assume that $m$ is a generator
for $P$.  Since $P$ is a projective $\gen{a}$-module,
$a.m \neq 0$.

All that remains to be shown is that if $w$ is a generator for the
$\Omega^t(k|_{\gen{a}})$ component of $\Omega^1(M)|_{\gen{a}}$, then
$w$ is a generator for $\Omega^1(M)$ and $b.w \neq 0$.

Let $0 \rightarrow \Omega^1(M) \rightarrow P
\xrightarrow{\psi} M \rightarrow 0$ be the exact sequence defining
$\Omega^1(M)$.  Let $v$ be a generator for the $\Omega^t(k|_{\gen{a}})$
component of $M|_{\gen{a}}$.  This generator can be chosen so that
$v = \psi(p)$ where $p$ is a generator of $P$.
We claim that
$a.p$ is a generator for $\Omega^1(M) \subseteq P$.  If not,
then write
$$
a.p = a.l + b.m
$$
for some elements $l$, $m \in \Omega^1(M)$.  Then
$$
b a.p = b  a.l.
$$
Define $\omega := b.p - b.l$.  Since $a.\omega = b.\omega = 0$,
by definition,
$\omega \in \Soc(P) = \tilde{\mf{a}}.P \subseteq \Omega^1(M)$.
So then $b.p = \omega + b.l \in \Omega^1(M)
	= \ker \psi$ (as in Definition \ref{D:syzygy}).  Now
$$
\psi(\omega + b.l) =  \psi(b.p) = b.\psi(p) = b.v \neq 0
$$
by assumption since $M$ satisfies (3).  This is a contradiction, thus,
$a.p$ is a generator for $\Omega^1(M)$.

Note that in $\Omega^1(M)|_{\gen{a}}$, $a.p$ generates a trivial
$\gen{a}$-module.  Thus, any generator for the $\Omega^1(k|_{\gen{a}})$
component of $\Omega^1(M)|_{\gen{a}}  \cong \Omega^t(k|_{\gen{a}})
\oplus Q$ is equivalent to $a.p$ modulo
$\Soc(Q) = a.Q$.
So if $w$ is any such generator, then it is possible to
write $w = c a.p + a.\nu$ for
some $0 \neq c \in k$ and $\nu \in \Soc(Q)$.  Then
$$
b.w = b.(c a.p + a.\nu) = c'\tilde{\mf{a}}.p \neq 0.
$$

These
computations show that condition (3) holds for $\Omega^1(M)$, hence
condition (1) does as well.  An inductive argument completes the proof of the
proposition.
\end{proof}

\begin{proposition} \label{P: inv or k}
Let $M$ be a $V(\mf{a}_2)$-supermodule which has no projective submodules
satisfying (1), (2), and
(3), then either $\Omega\inv(M)$ satisfies all three conditions or
$\Omega\inv(M)$ has a summand which is isomorphic to $k$ (either
even or odd).
\end{proposition}
\begin{proof}
Let $0 \rightarrow M \rightarrow I
\xrightarrow{\psi} \Omega\inv(M) \rightarrow 0$ be the exact sequence
of $V(\mf{a}_2)$-supermodules defining $\Omega\inv(M)$. By Lemma
\ref{L: 2 holds}, $\Omega\inv(M)$ already satisfies (2).  Recall that
because we are working in a self-injective category, the module $I$
is also projective and we can take advantage of the previous description
of these modules.

Let $v \in M$ be a generator for the $\Omega^t(k|_{\gen{a}})$
component of some
decomposition of $M|_{\gen{a}}$.  Our previous work shows that we may
choose $v$ so that $v = a.p$ for some $p \in I$.  We also know that
for some nonzero $c \in k$,
$c\tilde{\mf{a}}.p = b.v \neq 0$.  Then $p$ is a generator for $I$ and
$\psi(p)$ is a generator for $\Omega\inv(M)$.  There are two cases to
consider.

First, if $b.\psi(p) = 0$, then $\psi(p) \in \Soc(\Omega\inv(M))$,
since $a.\psi(p) = \psi(a.p) = \psi(v) = 0$ because $v \in M$.
If this happens, then $k \cdot \psi(p) \cong k$ is a direct summand
of $\Omega\inv(M)$.

For the rest of the proof, assume that
for any such element $p$, $b.p \notin M$ and we will
show that $\Omega\inv(M)$
satisfies all three conditions.  Note that, by Lemma
\ref{L: 2 then 4 iff 1}, we only need to establish (3).  Indeed, it
has been observed that $\psi(p)$ is a generator for $\Omega\inv(M)$ and
by the assumption, $b.p \notin M$ yields that
$b.\psi(p) \neq 0$ in $\Omega\inv(M)$.

Now let $m$ be a generator for $\Omega\inv(M)$ such that $a.m = 0$.
By (2), $\Omega\inv(M)|_{\gen{a}} \cong \Omega^t(k|_{\gen{a}}) \oplus P$
where $P$ is a projective $\gen{a}$-module.  Since
$a.\psi(p) = 0$, any generator of the $\Omega^t(k|_{\gen{a}})$
component will be equivalent to $\psi(p)$ modulo $\Soc(P)$.  This case
has been covered since we assumed $b.p \notin M$ so assume that $m$
does not generate the $\Omega^t(k|_{\gen{a}})$
component and, thus, must be a generator
for the projective summand.  However, since $a.m = 0$,
$m \in \Soc(P)$ and we conclude that $m$ cannot be
a generator for $P$, a contradiction.  So if
$m \not \equiv \psi(p) \bmod{ \Soc(P)}$ is any generator, then $a.m \neq 0$.
Thus, $\Omega\inv(M)$
satisfies (3) and thus (1) in this case, and the proof is complete.
\end{proof}

\begin{theorem} \label{T:Classification for rank 2}
Let $M$ be an endotrivial $V(\mf{a}_2)$-supermodule
in $\mc{F} $.  Then
$M \cong \Omega^n(k) \oplus P$ for some $n \in \Z$ and where $k$ is either
the trivial module $k_{ev}$ or $\Pi(k_{ev}) = k_{od}$ and $P$ is
a projective module in $\mc{F}$.
\end{theorem}
\begin{proof}
It has been observed that if $M$ is endotrivial, then $M$ satisfies
condition (2).  Additionally, for some $r \geq 0$,
$\Omega^r(M)$  satisfies (1).  So by Lemma
\ref{L: 2 then 4 iff 1}, $\Omega^r(M)$ satisfies (3) as well.

By (1), we can see that $\Omega^r(M)$ has no summand isomorphic to
$k$.  Assume that $\Omega^{-s}(\Omega^r(M))$ has no such summand
for for all $s >0$.  By Proposition \ref{P: inv or k},
$$
\Rk(M) > \Rk(\Omega\inv(M)) > \Rk(\Omega^{-2}(M)) > \cdots
$$
which is clearly impossible since $\Rk(M)$ is finite for any
module in $\mc{F}$.  Thus,
$\Omega^{-n-r}(\Omega^r(M)) \cong k \oplus Q$ for some $n \in \Z$.
Since $\Omega^{-n-r}(\Omega^r(M))$ satisfies (2),
$Q|_{\gen{a}}$ is a projective $\gen{a}$-module and
by considering the rank variety of $V(\mf{a}_2)$, $Q$ is a projective
$V(\mf{a}_2)$-module.
The $k$ summand may either be contained in $\ev{\Omega^{-n}(M)}$ or
$\od{\Omega^{-n}(M)}$ and since $\Omega^{-n}(M)$ contains no
projective submodules, $\Omega^{-n}(M) \cong k$ and
$\Omega^{0}(M) \cong \Omega^n(k)$.
By Lemma \ref{L: endo is indec plus proj},
\begin{equation*}
M \cong \Omega^n(k) \oplus P
\end{equation*}
where $P$ is a projective $V(\mf{a}_2)$-supermodule and $k$ is either
$k_{ev}$ or $k_{od}$.
\end{proof}

Given this theorem, it is now possible to identify the group $T(\mf{a}_2)$.

\begin{theorem} \label{T: T(g) for rank 2}
Let $\mf{a}_2$ be a rank 2 detecting subalgebra of $\mf{g}$,
then $T(\mf{a}_2) \cong \Z \times \Z_2$ and is generated
by $\Omega^1(k_{ev})$ and $k_{od}$.
\end{theorem}
\begin{proof}
Let $M$ be an endotrivial $V(\mf{a}_2)$-supermodule.  By Theorem
\ref{T:Classification for rank 2}, in the stable module category,
$M \cong \Omega^n(k_{ev})$ or
$M \cong \Omega^n(\Pi(k_{ev}))$.  By Lemma \ref{L: Pi is Omega invariant},
this can be rewritten as $M \cong \Omega^n(k_{ev})$ or
$M \cong \Pi(\Omega^n(k_{ev}))$.  Since the group operation in $T(\mf{a}_2)$
is tensoring over $k$ and by Proposition \ref{P:syzygy-operation}
(\ref{C: syzygy tensor syzygy}),
\begin{equation*}
M \cong
\begin{cases}
\sy{1}{k_{ev}}^{\otimes n} \otimes_k (k_{od})^{\otimes t}
	& \text{if $n > 0$}\\
\sy{-1}{k_{ev}}^{\otimes n} \otimes_k (k_{od})^{\otimes t}
	& \text{if $n < 0$} \\
\sy{1}{k_{ev}} \otimes_k \sy{-1}{k_{ev}} \otimes_k (k_{od})^{\otimes t}
	& \text{if $n = 0$}
\end{cases}
\end{equation*}
where $t \in \{ 1, 2\}$.  Thus, there is an isomorphism, $\phi$, from
$T(\mf{a}_2)$ to $\Z \times \Z_2$ given by
\begin{equation*}
\phi(M) :=
\begin{cases}
(n, t) & \text{if $M \cong \sy{1}{k_{ev}}^{\otimes n} \otimes_k (k_{od})^{\otimes t}$ for $n > 0$ }\\
(n, t) & \text{if $M \cong \sy{-1}{k_{ev}}^{\otimes n} \otimes_k (k_{od})^{\otimes t}$ for $n<0$} \\
(0, t) & \text{if $M \cong k_{ev} \otimes_k (k_{od})^{\otimes t} $}
\end{cases}
\end{equation*}
and it is now clear that $T(\mf{a}_2)$ is generated by $\sy{1}{k_{ev}}$
(and its inverse) and $k_{od}$.
\end{proof}

\section{Computing $T(\mf{g})$ for All Ranks Inductively}

Now we wish to proceed by induction to classify endotrivial modules
for the general case $\mf{a}_r$ where $r > 2$.
The structure of
$\mf{e}_n \cong \mf{q}(1) \times \dots \times \mf{q}(1) \subseteq \mf{gl}(n|n)$
and $\mf{f}_n \cong \sL \times \dots \times \sL \subseteq \mf{gl}(n|n)$
where there are $n$ copies of $\mf{q}(1)$ and $\sL$ respectively is
given in Section \ref{SS: rank r detecting}.

By Lemma \ref{L: endo in principal},
any endotrivial module for a detecting subalgebra
is in the principal block, so we consider (equivalently)
endotrivial representations of
$V(\mf{e}_n) = \Lambda(\od{(\mf{e}_n)})$
and $V(\mf{f}_n) = \Lambda(\od{(\mf{f}_n)})$.
The support variety theory
of \cite[Section 6]{BKN1-2006} will be used as well.  For an endotrivial
module $M$, since $M \otimes M^* \cong k_{ev} \oplus P$, we have
$\mc{V}_{(\mf{a}, \ev{\mf{a}})}(M) = \mc{V}_{(\mf{a}, \ev{\mf{a}})}(k_{ev})
\cong \A^r$ (which is also equivalent to the rank variety of $M$).

Our first step in the classification comes by following
\cite[Theorem 4.4]{CN-2009}.  Recall that
$\{a_1, \dots, a_r \}$ denotes a basis for $\od{(\mf{a}_r)}$  and that
$V(\mf{a}_r) = \langle 1, a_1, \dots, a_r \rangle$ when generated as
an algebra.

\begin{theorem} \label{T: shift rank is constant over rank 2 subalgebra}
Let $M$ be an endotrivial $V(\mf{a}_r)$-supermodule,
where $\mf{a}_r$ is a rank $r$ detecting subalgebra.  Let
$v = c_1 a_1 + \cdots + c_r a_r \in \od{(\mf{a}_r)}$ with $c_i \neq 0$
for some $i < r$ and let $A = \langle v, a_r \rangle$ be the subsuperalgebra of $V(\mf{a}_r)$
of dimension 4 generated by $v$ and $a_r$.  Then, for some $s$ independent
of the choice of $v$,
$M|_A \cong \sy{s}{k|_A} \oplus P$ for some $\mf{a}_r$-projective module $P$
where $k|_A$ is either the trivial module $k_{ev}$ or $\Pi(k_{ev})=k_{od}$.
\end{theorem}
\begin{proof}
First, note that since $\mf{a}_r$ is a purely odd, abelian Lie
superalgebra,
$v \otimes v = \frac{[v,v]}{2}=0$ 
and $a_r \otimes a_r = 0$ but $v \otimes a_r =-a_r \otimes v \neq 0$
and so $A \cong V(\mf{a}_2)$.  Also note
that if $v' = c_1 a_1 + \cdots c_{r-1} a_{r-1}$, since $c_i \neq 0$
for some $i<r$, then
$\langle v, a_r \rangle \cong \langle v', a_r \rangle$
by a change of basis.  So without loss
of generality, redefine $v = c_1 a_1 + \cdots c_{r-1} a_{r-1}$ and
$A = \langle v, a_r \rangle$ for the new $v$ and identify all such $v$
with the points in $\A^{r-1} \setminus \{ 0 \}$.

By the previous classification,
$\Omega^0(M|_A) \cong \sy{m_v}{k}$ where $k$ is either even or odd.  We now
show that $m_v$ is independent of the choice of $v$.

Since $\dim \sy{m}{k|_A} = \dim \sy{-m}{k|_A} > \dim M$ for
large enough $m$, then there exist $b$, $B \in \Z$ such that
$b \leq m_v \leq B$ for any $v \in \A^{r-1} \setminus \{ 0 \}$.  Moreover, we can
choose $b$ and $B$ such that equality holds for some $v$ and $v'$.  Now replace $M$ by $\sy{-b}{M}$.  Once this is done,
we assume $b=0$, and for all $v \in \A^{r-1}\setminus \{ 0 \}$,
$0 \leq m_v \leq B$ where the bounds are actually attained.

Let $C \in \Z$ be such that $0 \leq C < B$ and let
$$
S_C = \{ v \in \A^{r-1} \backslash \{0\} \ | \ m_v > C \}
$$
We claim that $S_C$ is closed in the Zariski topology of
$\A^{r-1}\setminus \{ 0 \}$.

Recall that, since we are working with $V(\mf{a}_2)$-modules, i.e.,
in the principal block, there is (up to the parity change functor)
a unique simple module, $k$, and
indecomposable projective module, which is
isomorphic to the left regular representation of $V(\mf{a}_2)$ (see
Sections \ref{SS: reps of detecting sub} and \ref{SS: T(a2)}).  Thus,
any projective module $P$ has dimension $4n$ for some $n \in \N$.

Since $m_v = 0$ for some $v$, it follows that $\dim M \equiv 1 \pmod{4}$.
This implies that $m_v$ is even for all $v$ since
$\dim \Omega^n(k|_A) = 1 + 2|n|$.  Thus,
for any $v$, $\dim \sy{2s}{k|_A} = 1 + 4s$ for $s \geq 0$.  With
this in mind, define
$$
t = (\dim M - \dim \sy{2c}{k|_A})/4
$$
where $c = C/2$ if $C$ is even and $c = (C-1)/2$ if $C$ is odd.  In
either case, $2c \leq C < 2c+2$.  This construction is done to ensure
that for any $v$, the statement that $m_v \leq C$ means that the
dimension of the projective part of $M|_A$ is
$$
\dim M - \dim \sy{m_v}{k|_A} \geq 4t.
$$
In other words, if $m_v \leq C$, then $M|_A$ has an $A$-projective summand
of rank at least $t$ so the rank of the matrix of the element
$\omega_v = v \otimes a_r$ (which generates the socle of $A$)
acting on $M$ is at least
$t$.  Otherwise, if $m_v > C$, then $M|_A$ has no
$A$-projective summand of rank $t$.  Consequently, the rank of
the matrix of $\omega_v$ is strictly less than $t$.

Let $d = \dim M$ and let $\mc{S}$ be the set of all subsets of
$\mf{N} = \{1, \dots , d \}$ having exactly $t$ elements.  For any
$S$, $T \in \mc{S}$ define $f_{S,T} : \A^{r-1}\setminus \{ 0 \} \rightarrow k$ by
$$
f_{S,T}(v) = \operatorname{Det}(M_{S,T}(\omega_v))
$$
where $M_{S,T}$ is the $t \times t$ submatrix of the matrix of $\omega_v$
acting on $M$ having rows indexed by $S$ and columns indexed by $T$.
The functions $f_{S,T}$ are  polynomial
maps and so their common set of zeros $\ms{V}(\{ f_{S,T} \}_{S,T \in \mc{S}})$ is
a closed set of $\A^{r-1} \setminus \{0\}$.
If $M|_A$ has no $A$-projective summand of rank $t$,
then each determinant must always be 0, hence in the vanishing locus,
and otherwise, at least one
of the $f_{S,T}(v)$ will be nonzero.  Thus we have constructed
a set of polynomials
such that $f_{S,T}(v) = 0$ on each polynomial $f_{S,T}$
if and only if $v \in S_C$.  We
conclude that $S_C$ is closed in $\A^{r-1}\setminus \{ 0 \}$.

It is also true that for any $C$, $S_C$ is open in $\A^{r-1}\setminus \{ 0 \}$.
First, replace $M$ with $M^*$ (which is also endotrivial).
Since $(\sy{n}{M^*})^* \cong \sy{-n}{M}$, for $M^*$, the bounds are
$-B \leq m_v \leq 0$.  Replacing $M^*$ with $\sy{B}{M^*}$ again yields
$0 \leq m_v \leq B$ .  However, now we have that for any $v$,
$$
M|_A \cong \sy{m_v}{k|_A} \oplus P,
$$
and by the above computation, we also have
$$
(\sy{B}{M^*})|_A \cong \sy{B - m_v}{k|_A} \oplus P.
$$
Thus, $S_C = (S_{B-C})^c$ and so $S_C$ is open.  Since $S_C$ is both
open and closed and $\A^{r-1}\setminus \{ 0 \}$ is connected, we conclude that
$S_C$ is either the empty set or all of $\A^{r-1}\setminus \{ 0 \}$.  By assumption, 
there is a $v$ such that $m_v = 0$, so $S_0$ is nonempty.  Thus,
$S_0 = \A^{r-1} \setminus \{0\}$ and
$B = 0$ as well (since the bounds are attained).
Thus, the number
$m_v$ is constant over all $v \in \A^{r-1}\setminus \{ 0 \}$ and
$$
M|_A \cong \sy{s}{k|_A} \oplus P
$$
for any subsuperalgebra $A \cong V(\mf{a}_2)$ where $k$ is either
$k_{ev}$ or $k_{od}$, by the classification of $T(\mf{a}_2)$.

We also claim that for any such $A$, the parity of $k|_A$ is constant
as well.  This can be seen by assuming that there are $A$ and $A'$ such that
$M|_A \cong \sy{s}{k_{ev}} \oplus P$ and
$M|_{A'} \cong \sy{s}{k_{od}} \oplus P'$.  Now consider the dimensions
of $\ev{M}$ and $\od{M}$.  Since $\dim \sy{s}{k_{ev}} =
\dim \sy{s}{k_{od}}$, it follows that $\dim P = \dim P'$.  Note
that $\dim \ev{P} = \dim \od{P}$ and consequently,
$\dim \ev{P} = \dim \ev{P'}$ and $\dim \od{P} = \dim \od{P'}$.
Finally, recall that
$\dim \ev{\sy{s}{k_{ev}}} \neq \od{\sy{s}{k_{ev}}}$.  Without
loss of generality, assume that
$\dim \ev{\sy{s}{k_{ev}}} > \dim \od{\sy{s}{k_{ev}}}$, i.e. $s$ is an
even integer.  Then
$$
\dim \ev{\sy{s}{k_{od}}} =  \dim \ev{\sy{s}{\Pi(k_{ev})}} =
	\dim \Pi(\ev{\sy{s}{k_{ev}}}) = \dim \od{\sy{s}{k_{ev}}}
$$
and similarly,
$$
\dim \od{\sy{s}{k_{od}}} =  \dim \od{\sy{s}{\Pi(k_{ev})}} =
	\dim \Pi(\od{\sy{s}{k_{ev}}}) = \dim \ev{\sy{s}{k_{ev}}}.
$$
This implies that $\dim \ev{\sy{s}{k_{od}}} <
\dim \od{\sy{s}{k_{od}}}$.
These different decompositions combine to yield that
$\dim \ev{M} > \dim \od{M}$ by considering
$M|_A$ and $\dim \ev{M} < \dim \od{M}$ by considering $M|_{A'}$.
This is a contradiction and so the parity of the $k$ in the decomposition
of $M|_A$ is constant for any choice of $A$ as well.
\end{proof}

\begin{theorem}
Let $M$ be an endotrivial $V(\mf{a}_r)$-supermodule, then
$M \cong \Omega^n(k) \oplus P$ for some $n \in \Z$ where $k$ is either
the trivial module $k_{ev}$ or $\Pi(k_{ev}) = k_{od}$ and $P$ is
a projective module in $\mc{F}$.
\end{theorem}
\begin{proof}
Let $M$ be an endotrivial $V(\mf{a}_r)$-supermodule and let $A= 
\langle v, a_r \rangle$ where
$v = c_1 a_1 + \cdots + c_{r-1} a_{r-1}$ for some
$(c_1, \dots, c_{r-1}) \in \A^{r-1} \setminus \{0\}$.  By
Theorem \ref{T: shift rank is constant over rank 2 subalgebra},
$M|_A \cong \sy{m}{k|_A} \oplus P$ and $m$ is
independent of the choice of $v$.
The goal is to prove that
$M \cong \sy{m}{k} \oplus Q'$ or, equivalently, $\sy{-m}{M} \cong k
\oplus Q$.
For simplicity, replace $M$ by $\sy{-m}{M}$ and assume that $M|_A
\cong k|_A \oplus P$.

The first step is to show that the module $\hat{M} = a_r.M$ is a projective
$\hat{V} = V(\mf{a}_r)/(a_r)$ module.
We do this by considering the
rank variety $\mc{V}_\mf{a}^{\text{rank}}(\hat{M}|_{\hat{V}})$ (see
\cite[Section 6.3]{BKN1-2006}).

As in the previous proof, we are working in
the principal block and so, there is (up to the parity change
functor) a unique indecomposable projective $V(\mf{a}_r)$-module, which is
isomorphic to the left regular representation of $V(\mf{a}_r)$
in itself (see Section \ref{SS: reps of detecting sub}).  Note that
the dimension of these projective indecomposable modules is $2^r$.

Recall that $A=\langle v, a_r \rangle \cong V(\mf{a}_2) = \gen{a_1, a_2}$
and we assume that $M|_A \cong k|_A \oplus P_A$ where
$P_A$ is a projective $A$-module.  Then
$ a_r.M|_A \cong a_r.k|_A \oplus a_r.P_A \cong a_r.P_A$.  
The action of $a_r$ on these modules is trivial, so think of them now
as $v$-modules.  We also know that $a_r.P_A$ is still projective
as a $v$-module, since $a_2.V(\mf{a}_2) \cong V(\mf{a}_{1})$ as
$V(\mf{a}_{1})$-modules.
Hence $\hat{M}|_v = a_r.M|_v$ is projective
for all $v \in \A^{r-1} \setminus \{0\}$. This tells us that
$\mc{V}_\mf{a}^{\text{rank}}(\hat{M}|_{\hat{V}})= \{ 0 \}$ and so
$\hat{M}$ is a projective $\hat{V}$-module.

The projective indecomposable modules in the principal block are the
projective covers (or equivalently injective hulls) of the trivial
modules $k_{ev}$ and $k_{od}$.  Consequently, the simple, one dimensional
socle of $\hat{V}\cong V(\mf{a}_{r-1})$ is generated
by $\tilde{\mf{a}}_{r-1} = a_1 \otimes \cdots \otimes a_{r-1}$.  Thus,
$\dim \tilde{\mf{a}}_{r-1}.\hat{M}$ counts the number of summands of
projective $V(\mf{a}_r)$-modules in $\hat{M}$ and so
$$
\dim \hat{M} = 2^{r-1} \dim \tilde{\mf{a}}_{r-1}.\hat{M}.
$$
Also, $\tilde{\mf{a}}_r =\tilde{\mf{a}}_{r-1}\otimes a_r$ is a generator for
the socle of $V(\mf{a}_r)$ and $\tilde{\mf{a}}_{r-1}.\hat{M}
=\tilde{\mf{a}}_r.M$
by construction.  Therefore, $M$ has a projective submodule, $Q$ of
dimension $2^r \dim \tilde{\mf{a}}_r.M =
2^r \dim \tilde{\mf{a}}_{r-1}.\hat{M}$. Thus,
$$
2 \dim \hat{M} = \dim M - 1
$$
and we conclude that $M \cong k \oplus Q$.  Note, since this is
a direct sum decomposition, as super vector spaces,
$k = k|_A$.  Thus $k$ has the same
parity as $k|_A$ (which was uniquely determined by $M$) and
the claim is proven.
\end{proof}

We can now classify endotrivial $\mf{a}_r$-modules for all $r$.

\begin{theorem} \label{T: T(g) for rank r}
Let $\mf{a}_r$ be a rank r detecting subalgebra
where $r \geq 2$, then $T(\mf{a}_r) \cong \Z \times \Z_2$ and is generated
by $\Omega^1(k_{ev})$ and $k_{od}$.
\end{theorem}
\begin{proof}
The proof is exactly the same as in Theorem \ref{T: T(g) for rank 2}.
\end{proof}

\section{A Finiteness Theorem for $T(\mf{g})$}
Let $\mf{g} = \supalg{\mf{g}}$ be a classical Lie superalgebra.  Just as
in the case of finite group schemes, it is not known if the group of endotrivial
modules $T(\mf{g})$ is
finitely generated, but we can show that
in certain cases, there are finitely many
endotrivial modules of a fixed dimension $n$.

The main result of this section is Theorem \ref{T: fin many endo of dim n} which will
be proved at the end of Section
\ref{SS: variety of n dim rep}
by extending a proof in \cite{CN-2011} to Lie superalgebras.

\begin{theorem} \label{T: fin many endo of dim n}
Let $n \in \N$ and $\mf{g}=\supalg{\mf{g}}$ be a classical Lie superalgebra
such that there are finitely many simple modules of dimension $\leq n$ in
$\mc{F}_{\rel{g}}$.  Then there are only finitely many isomorphism classes
of endotrivial modules in $\mc{F}_{\rel{\mf{g}}}$ of dimension $n$.
\end{theorem}

In order to achieve a situation which is analogous to that in
\cite{CN-2011}, we must work in a module
category which is Morita equivalent to the category of
$\mc{F}_{\rel{g}}$ modules.  The issue is that projective $\mc{F}_{\rel{g}}$
modules do not appear as direct summands of $\U{g}$ nor does it contain
idempotents.
The equivalence is used to work in a module category over a finitely generated
algebra which satisfies
both conditions.

\subsection{Morita Equivalence}
Let
$Y^+$ denote a set which indexes the simple modules in $\mc{F} := \mc{F}_{\rel{g}}$ and
$L(\lambda)$ denote the simple module corresponding to $\lambda \in Y^+$ and
$P(\lambda)$ its projective cover.  Consider the following algebra.

\begin{definition}
Let $\mf{g} =\supalg{\mf{g}}$ be a classical Lie superalgebra.  Then define 
$$
A_{\mf{g}} :=
\End_{U(\mf{g})}^{fin} \left( \bigoplus_{\lambda \in Y^+} P(\lambda)^{\oplus \dim L(\lambda)} \right)^{\text{op}}
$$
which we will call the Khovanov algebra associated to $\mf{g}$.  The
notation $\End_{U(\mf{g})}^{fin}$ is to indicate that the endomorphisms
are $U(\mf{g})$-module endomorphisms which are supported on a finite
number of summands.

\end{definition}
By \cite{BS-2012}, the category of modules $\mc{F}$ is Morita
equivalent to the category of $A_\mf{g}$-modules, finite dimensional
modules correspond to finite dimensional modules of the same dimension, and
by construction, $A_\mf{g}$ contains idempotents, denoted as follows.
Let $\lambda \in Y^+$ and $d=d_\lambda = \dim L(\lambda)$.  Then $e_\lambda$ denotes the
idempotent associated to the identity endomorphism of
$\End_{\U{g}} (P(\lambda)^{\oplus d_\lambda})$ and
$e_\lambda = e_{\lambda_1} + \dots + e_{\lambda_d}$ where the $e_{\lambda_i}$
are the idempotents of each of the $d_\lambda$ summands of $P(\lambda)$.

Shifting to the representation theory of $A_\mf{g}$ is still not
a sufficient reduction since $A_\mf{g}$
is not finitely generated.
As indicated in Theorem \ref{T: fin many endo of dim n}, the Lie superalgebra
$\mf{g}$ must also satisfy the further restriction of having
finitely many simple modules of dimension $\leq n$ in $\mc{F}$ (a condition
which is further discussed in Section \ref{S: conditions on g}).  In this situation, given
that we are only interested in modules of dimension $\leq n$, it is sufficient to
consider a finite dimensional quotient of $A_\mf{g}$.

\begin{proposition} \label{P: factors through finite}
Let $\mf{g} = \supalg{\mf{g}}$ be a classical Lie superalgebra
such that there are finitely many simple modules of dimension $\leq n$ in
$\mc{F}_{\rel{g}}$, and let $A_\mf{g}$ be the Khovanov algebra associated to $\mf{g}$.
Let $M$ be an $A_\mf{g}$-module of dimension $\leq n$ and
$\rho : A_\mf{g} \rightarrow \mf{gl}(M)$ be the representation given by $M$.  Then
$\rho$ factors through a finite dimensional quotient of $A_\mf{g}$
\begin{equation*} %\label{E: }
%\begin{center}
\begin{tikzpicture}[description/.style={fill=white,inner sep=2pt},baseline=(current  bounding  box.center)]

\matrix (m) [matrix of math nodes, row sep=3em,
column sep=2.5em, text height=1.5ex, text depth=0.25ex]
{ A_\mf{g} & 	& \bigslanted{A_\mf{g}}{I_n} \\
	&  	& \mf{gl}(M) \\ };

\path[->,font=\scriptsize]
	(m-1-1) edge node[auto] {$ \pi $} (m-1-3)
			edge node[below] {$ \rho $} (m-2-3);
\path[dashed, ->, font=\scriptsize]
	(m-1-3) edge node[auto] {$ \overline{\rho} $} (m-2-3);
\end{tikzpicture}
%\end{center}
\end{equation*}
where $I_n$ is an ideal which depends only on $n$.
\end{proposition}
\begin{proof}
First, note that $A_\mf{g}$ can be given an identity
$e = \sum\nolimits_{\lambda \in Y^+} e_\lambda$ which is well defined since all other
endomorphisms are supported on a finite number of summands.
%Let $e_\lambda$ be the idempotent associated to the identity of $\End_{\U{g}} (P(\lambda)^{\oplus d_\lambda})$ and
Let $L_\mf{g}(\lambda)$ be the
simple $A_\mf{g}$-module which corresponds to $L(\lambda)$ under the Morita equivalence
of $\mc{F}$ and $A_\mf{g}$.  Note that $d_\lambda : = \dim L(\lambda) = \dim L_\mf{g}(\lambda)$ by the
construction of $A_\mf{g}$.

As detailed in \cite[Section 2]{DEN-2004}, if $M$ is an $A_\mf{g}$-module, then
$e_\lambda M \neq 0$ if and only if $[M : L_\mf{g}(\lambda)] \neq 0$.  Thus, if
$e_\lambda M \neq 0$, then $\dim M \geq \dim L_\mf{g}(\lambda)$.  So if $\dim M = n$
then for all $\mu \in Y^+$ with $\dim L(\mu) > n$, $e_\mu M = 0$ or
equivalently, if $\rho : A_\mf{g} \rightarrow \mf{gl}(M)$ is the representation given
by $M$, then $\rho(e_\mu) = 0$.

Given the assumption that there are finitely many simple modules of dimension $\leq n$
in $\mc{F}$, there are finitely many idempotents $e_\lambda$ such that
$\rho(e_\lambda) \neq 0$.

Next, $A_\mf{g}$ is generated as an algebra by the idempotents $e_\lambda$ for
$\lambda \in Y^+$ and homomorphisms
$$P(\lambda)^{\oplus d_\lambda} \rightarrow P(\mu)^{\oplus d_\mu}$$
which take the head of $P(\lambda)^{\oplus d_\lambda} = 
L(\lambda)^{\oplus d_\lambda}$ to the second layer of
$P(\mu)^{\oplus d_\mu}$.  The number of such homomorphisms for
any fixed $\lambda$ and $\mu$ is finite and given by
$\dim \Ext^1_{A_\mf{g}}(L_\mf{g}(\mu), L_\mf{g}(\lambda))$.

Finally, let $\alpha : P(\lambda)^{\oplus d_\lambda} \rightarrow P(\mu)^{\oplus d_\mu}$ be a generating element of $A_\mf{g}$ where $\rho(e_\mu) = 0$.
Then $e_\mu \circ \alpha = \alpha$ and so in $A_\mf{g}$, $\alpha \cdot e_\mu = \alpha$.
Since $\rho$ is a homomorphism of superalgebras,
$$
\rho(\alpha) = \rho( \alpha \cdot e_\mu) = \rho(\alpha) \cdot \rho(e_\mu) = 0,
$$
and similarly, if $\rho(e_\lambda) = 0$ then
$$
\rho(\alpha) = \rho(e_\lambda \cdot \alpha) =  \rho(e_\lambda) \cdot \rho(\alpha) = 0
$$
as well.  Thus, the support of $\rho$ is contained in $e_\lambda$ and
$\alpha : P(\lambda)^{\oplus d_\lambda} \rightarrow P(\mu)^{\oplus d_\mu}$ for
which $d_\lambda, d_\mu \leq n$, and the number of such $\lambda, \mu \in Y^+$
is, by assumption, finite.  So if $e_\lambda  = e_{\lambda_1} + \dots + e_{\lambda_d}$
where the $e_{\lambda_i}$
are the primitive idempotents for each of the $d=d_\lambda$ endomorphism algebras
$\End_{\U{g}}(P(\lambda))$ and
$$
I_n := \gen{e_{\lambda_i} \st \text{for $1 \leq i \leq d_\lambda$ where $d_\lambda > n$}  }
$$
then $A_\mf{g}/I_n$ is finite dimensional and $\rho$ factors through this
quotient for any representation of a module $M$ of dimension $\leq n$.
\end{proof}

Next, a lemma similar to that of \cite[Lemma 2.1]{CN-2009} is proven
for the algebra $A_\mf{g} / I_n$ which will be required in the following section.

\begin{lemma} \label{L: element which identifies projectives}
Let $\mf{g}=\supalg{\mf{g}}$ be a classical Lie superalgebra and let $A_\mf{g}$ be the
Khovanov algebra associated to $\mf{g}$ and $M$ be an $A_\mf{g}$-module.
Then there exist elements $u_\lambda \in A_\mf{g}$
such that, if $u_\lambda M \neq \{ 0\}$ then $M$ has a direct summand
isomorphic to $P_\mf{g}(\lambda)$.
Moreover, if
$b_\lambda = \dim (u_\lambda P_\mf{g}(\lambda))$ is the rank of the operator of
left multiplication by $u_\lambda$ on $P_\mf{g}(\lambda)$ and
$a_\lambda = \dim(u_\lambda M)/b_\lambda$, then
$M \cong P_\mf{g}(\lambda)^{a_\lambda} \oplus N$ where $N$
has no direct summands
isomorphic to $P_\mf{g}(\lambda)$
Furthermore, if $\dim M \leq n$ then the same result
holds in $A_\mf{g} / I_n$.
\end{lemma}
\begin{proof}
Since $A_\mf{g}$ is an algebra with idempotent decomposition, we can
assume that each projective module $P_\mf{g}(\lambda) = A_\mf{g} e_\lambda$ for an
idempotent $e_\lambda$.  Since $A_\mf{g}$ is self injective, each
projective indecomposable module has a simple socle and is thus generated
by any nonzero element $u_\lambda \in \Soc(P_\mf{g}(\lambda))$, thus
$A_\mf{g}u_\lambda = \Soc(P_\mf{g}(\lambda))$ and
$u_\lambda = u_\lambda e_\lambda$.

By assumption, there exists an $m \in M$ such that $u_\lambda m \neq 0$.
Define $\psi_\lambda : P_\mf{g}(\lambda) \rightarrow M$ by
$\psi_\lambda(ae_\lambda) = ae_\lambda m$ for $a \in A_\mf{g}$.  Note that
$\psi_\lambda(u_\lambda e_\lambda) = u_\lambda e_\lambda m = u_\lambda m
\neq 0$ and since $u_\lambda$ generates the socle of $P_\mf{g}(\lambda)$, the
map $\psi_\lambda$ is injective.  Furthermore, since $A_\mf{g}$ is self
injective, this projective module is injective and so $M \cong P_\mf{g}(\lambda)
\oplus N$.  The multiplicity of the projective module follows from
an inductive argument applied to the module $N$.

By Proposition \ref{P: factors through finite}, if $\dim M \leq n$ then the action
$\rho$ of $A_\mf{g}$ on $M$ has $\rho(I_n) = 0$ an
thus induces an action of $A_\mf{g} / I_n$ on $M$ via $\ov{\rho}$.
Similarly, if $M$ has any projective summands, then they are of dimension $\leq n$ and
the $\ov{\rho}$ action is that of $\rho$ and it follows that each argument above
holds in $A_\mf{g} / I_n$ as well.
\end{proof}

\subsection{The Variety of $n$ Dimensional Representations} \label{SS: variety of n dim rep}
Since the remainder this section is only concerned with $A_\mf{g}$-modules of dimension
$n$ for a fixed integer $n$, the algebra $A_\mf{g} / I_n$ will be denoted as
$A_{n}$ and the $\ov{\rho}$ action will be assumed.

Now we turn to the variety of $n$ dimensional $A_{n}$
representations.  This variety is a construction of
Dade, as introduced in \cite{Dade-1979}.  The goal is to show that
the subvariety of endotrivial modules is open and that each component has
a finite number of isomorphism classes of endotrivial modules in it.
Since there is a correspondence between $n$ dimensional $A_{n}$-modules
and $n$ dimensional modules in $\mc{F}$, the same result holds in $\mc{F}$.

Now, we actually construct the variety of $n$ dimensional representations
with some additional structure to account for the $\Z_2$ grading of the
representations.
The variety of all representations of a fixed
dimension $n$ is denoted $\mc{V}_n$ and is defined by considering
a set of homogeneous generators $g_1, \dots, g_r$ for the superalgebra $A_{n}$.

Different notation is used in this section to distinguish the treatment of
representations from the previous section.
A representation is a homomorphism of superalgebras, and is denoted
$\varphi : A_{n} \rightarrow \End_k (V)$ where $\dim(V)=n$, and
if a homogeneous basis for $V$ is fixed, we can think of this homomorphism as
a superalgebra homomorphism $\varphi : A_{n} \rightarrow M_n(k)$.
Since $g_1, \dots, g_r$ generate $A_{n}$, the map $\varphi$ is completely
determined by $\varphi(g_i)$ which is an $n \times n$ matrix with entries
$(g_{i,st})$ in $k$ where $1 \leq i \leq r$ and $1 \leq s, t \leq n$.

Consider the polynomial ring $R = k[x_{i,st}]$,
where $1 \leq i \leq r$ and $1 \leq s, t \leq n$, which
has $rn^2$ variables.  The information of each
representation can be encoded in the form of a variety by defining a map
$\overline{\varphi} : A_{n} \rightarrow M_n(R)$ by
$(\overline{\varphi}(g_i))_{st} = x_{i,st}$ for $1 \leq i \leq r$.  Since
$A_{n} = \gen{g_1, \dots, g_r} / \mc{I}$,
the relations in $\mc{I}$ must be imposed on $M_n(R)$
by constructing the following ideal
of $R$.  By using $\overline{\varphi}$, a relation
in $\mc{I}$ is transferred to the same relation in $M_n(R)$ by creating a
relation on the rows and columns in the corresponding matrix multiplication.
For example, take the relation $g_1g_2 = 0$ in $A_{n}$.
This would correspond to the relation
$\sum_{u=1}^n x_{1,su}x_{2,ut}$ = 0 for each $1 \leq s, t \leq n$.
With these relations,
the matrices have the same algebra structure as $A_{n}$ does
(and now $\ov{\varphi}$ is actually a homomorphism), but
they are expressed as zero sets in the polynomial ring $R$.  Thus,
the ideal $\mc{I}$ uniquely corresponds to an ideal $\mc{J} \subseteq R$.
If $\mc{V}_n := \mc{V}(\mc{J}) \subseteq k^{rn^2}$, then
each point of $\mc{V}_n$ uniquely defines a representation of
$A_{n}$ by fixing the entries in a matrix in $M_n(k)$.

Additionally, the representations are required to be superalgebra
homomorphisms which means that the map $\ov{\varphi}$ is degree $\ov{0}$.
The
representations encoded in this variety all have dimension $n$
but there is no restriction on the dimensions of the even and odd
component of the representation, hence the image of an even or odd element
of $A_{n}$ is not restricted either.  Thus, no particular
$x_{i,st}$ is associated with a grading except when $s=t$ which must
necessarily correspond to an even element of $A_{n}$.

Many isomorphic representations are encoded multiple
times via a change of base.  However, since we are considering
super representations, the isomorphisms must preserve the grading
of the underlying vector space.  If the dimensions
of the even and odd component of the module are $s$ and $t$
where $s + t = n$,
the orbit of the point corresponding to the module under conjugation
by an element of $\ev{(GL_n(s|t))}$
will yield an isomorphic representation.
Unfortunately, since these orbits are given by conjugation of
different matrices, the grading information is not encoded well in
this variety and so we are content to note that
$\mc{V}_n$ contains all $n$ dimensional
representations of $A_{n}$ as points.
This observation will be sufficient
for the remainder of this section.

Lastly, in order
to achieve the results here, we artificially introduce a tensor structure
on $A_\mf{g}$, and similarly $A_{n}$,
by considering the functor which gives the Morita equivalence
between $\mc{F}$ modules and $A_\mf{g}$-modules.
According to \cite[Section 5]{BS-2012}, the functor is given by
\begin{equation} \label{E: morita functor}
F(-) := \Hom_{\U{g}} \left(\bigoplus_{\lambda \in Y^+} P(\lambda)^{\oplus \dim L(\lambda)}, - \right)
\end{equation}
and then $A_\mf{g}$ acts by precomposition.

Let $F'$ denote the inverse
adjoint equivalence of $F$
(such an equivalence exists by \cite[Theorem 4.2.3]{HTT-2013}).
Then for 
$A_\mf{g}$-modules, $M$ and $N$, define
\begin{equation*}
M \otimes N := F(F'(M) \otimes F'(N)).
\end{equation*}
Note that by
construction, $F$ is now a tensor functor.

Furthermore, if $M$ and $N$ are instead $A_{n}$-modules and
$\dim M \cdot \dim N \leq n$ then $M$ and $N$ lift to $A_\mf{g}$ modules and the same
definition can be used to define a tensor product on $A_{n}$.

Given this setup, we may now use the variety of $n$ dimensional
$A_{n}$-modules and the functor $F$ (Equation \ref{E: morita functor}),
to prove corresponding versions of
\cite[Lemmas 2.2 and 2.3]{CN-2011} and
\cite[Theorem 2.4]{CN-2011}.

\begin{lemma}
Let $M$ be an $n$ dimensional $A_{n}$-supermodule, $P_\mf{g}(\lambda)$ a
projective indecomposable
$A_{n}$-supermodule and $m \in \N$.  Let $\mc{U}$ be the subset of $\mc{V}_n$
of all representations, $\overline{\varphi}$, of $A_{n}$ such that
$M \otimes L_{\overline\varphi}$ has no submodule isomorphic to
$P_\mf{g}(\lambda)^m$, where
$L_{\overline{\varphi}}$ is the module given by the representation
$\overline{\varphi}$.  Then $\mc{U}$ is closed in $\mc{V}_n$.
\end{lemma}
\begin{proof}
Since the tensor product of two $n$ dimensional modules is
considered in this lemma, we temporarily work in another
variety $\mc{V}'_n$ whose construction is analogous to the one
above.

Let $A_{n^2} = \gen{g_1, \dots, g_r,
\dots, g_{r'}}/\mc{I}'$ and $R' := k[x_{i,st}]$ where $1 \leq i \leq r'$ and $1 \leq s, t \leq n$.  Then define
$\ov{\varphi}': A_{n^2}
\rightarrow M_n(R')$ as above and consider
the ideal $\mc{J}' \subseteq R'$ corresponding to $\mc{I}'$.
Then $\mc{V}'_n := \mc{V}(\mc{J}') \subseteq k^{r'n^2}$.  This is an enlarged version of
$\mc{V}_n \subseteq \mc{V}'_n$ which is contained canonically
as the closed subvariety given by $x_{i,st} = 0$ for $i > r$.
Thus any $\ov{\varphi} \in \mc{V}_n$ can be
thought of as being in $\mc{V}'_n$ and $L_{\overline\varphi} =
L_{\overline{\varphi}'}$ under the action of $A_n$ and the
action of the other generators of $A_{n^2}$ on
$L_{\overline{\varphi}'}$ is necessarily zero.

Using a similar idea to one in the proof of
Theorem \ref{T: shift rank is constant over rank 2 subalgebra},
consider the rank of the matrix of $u_\lambda \in A_{n^2}$ (defined and constructed in
Lemma \ref{L: element which identifies projectives}) on
$M \otimes L_{\overline{\varphi}'}$.  Denote this matrix by $M_{u_\lambda}$
and let $t$ denote the rank of the matrix of $u_\lambda$ acting
on $P_\mf{g}(\lambda)$.  By Lemma
\ref{L: element which identifies projectives},
$M \otimes L_{\overline{\varphi}'}$ will have a submodule
isomorphic to $P_\mf{g}(\lambda)^m$ if and only if the rank
of $M_{u_\lambda}$ is at least $mt$.
Since $u_\lambda$ is a polynomial
in the generators of $A_{n^2}$ and
for a fixed representation $M$ the matrix of the action of
$\overline{\varphi}'(g_i)$ on $M \otimes L_{\overline{\varphi}'}$ has entries in the
polynomial ring $R'$,
the entries of $M_{u_\lambda}$ are also all polynomials in $R'$.

By the same reasoning in Theorem
\ref{T: shift rank is constant over rank 2 subalgebra}, the condition
that the rank of $M_{u_\lambda}$ be less than $mt$ is the same
as the condition
that any $mt \times mt$ submatrix have determinant zero which we can
then be translated into a condition that certain
polynomials in $R'$ be zero.  These relations define a
closed subset $\mc{U}'$ in $\mc{V}'_n$.

Since $\mc{V}_n$ was artificially extended
to account for the action of 
$A_{n^2}$ on the tensor product of modules, we take
$\mc{U} := \mc{U}' \cap \mc{V}_n$ which is by construction the subset of
$\mc{V}_n$ such that
$M \otimes L_{\overline\varphi}$ has no submodule isomorphic to
$P_\mf{g}(\lambda)^m$ and has now been shown to be closed.
\end{proof}

\begin{lemma}
Let $M$ be an endotrivial module in $\mc{F}_{\rel{g}}$ of dimension $n$.  Let
$\mc{U}$ be the subset of representations $\ov{\varphi}$ of $\mc{V}_n$ such
that $L_{\overline{\varphi}}$ is not isomorphic to $F(M) \otimes  \lambda$
for any one dimensional module $\lambda$, where $L_{\overline{\varphi}}$
is the module given by $\overline{\varphi}$.  Then $\mc{U}$ is closed
in $\mc{V}_n$.
\end{lemma}
\begin{proof}
Let $\mu$ be a one dimensional module in $\mc{F}$.  Since $M$ is
endotrivial, so is $M \otimes \mu$.  Then since $F$ is a tensor
functor,
$$
F((M \otimes \mu) \otimes (M \otimes \mu)^*) \cong
k \oplus \bigoplus_{i = 1}^l P_\mf{g}(\lambda_i)^{n_i}
$$
where $P_\mf{g}(\lambda_i)$ is a projective indecomposable
$A_{n}$-module and $n_i \in \N$.
For each $i$, let $\mc{U}_i \subseteq \mc{V}_n$ where
$$
\mc{U}_i := \{ \ov{\varphi} \in \mc{V}_n \st 
	L_{\overline{\varphi}} \otimes F(M^*) \otimes \mu^* \text{ does not contain
	a submodule isomorphic to $P_\mf{g}(\lambda_i)^{n_i}$} \}.
$$
By the previous lemma,
each $\mc{U}_i$ is closed and so is
$\mc{U}_\mu = \mc{U}_1 \cup \cdots \cup \mc{U}_l$.

Clearly, for any
$\ov\varphi \in \mc{U}$, $L_{\overline{\varphi}}$ is not isomorphic to
$F(M) \otimes \mu$ since they have different projective
indecomposable summands.  Now we will consider
some $\ov\varphi \notin \mc{U}_\mu$ and show that $L_{\overline{\varphi}} \cong
F(M) \otimes \lambda$ for some one dimensional module $\lambda$.
Since $\ov\varphi \notin \mc{U}_\mu$,
$$
L_{\overline{\varphi}} \otimes F(M^*) \otimes \mu^*
\cong \nu \oplus \bigoplus_{i = 1}^l P_\mf{g}(\lambda_i)^{n_i}
$$
for some supermodule $\nu$.  However,
$\dim L_{\overline{\varphi}} \otimes F(M^*) \otimes \mu^* = n^2
= \dim (F(M) \otimes \mu) \otimes (F(M^*) \otimes \mu^*)$ so
$\nu$ is one dimensional.  It then follows that
$$
\nu^* \otimes L_{\overline{\varphi}} \otimes  F(M^*) \otimes \mu^*
\cong k \oplus \bigoplus_{i = 1}^l (\nu^*\otimes P_\mf{g}(\lambda_i)^{n_i}).
$$
Tensoring both sides by $\nu \otimes F(M) \otimes \mu$ yields
$$
L_{\ov{\varphi}} \oplus \bigoplus_{i = 1}^l P_i^{n_i}
 \cong (F(M) \otimes \mu \otimes \nu) \oplus \bigoplus_{i = 1}^l (P_i^{n_i}
 \otimes F(M) \otimes \mu)
$$
and so by comparing the nonprojective summands,
$L_{\overline{\varphi}} \cong F(M) \otimes \nu \otimes \mu 
\cong F(M) \otimes \lambda$ where $\lambda = \nu \otimes \mu$.  Then
$\mc{U}_\mu$ is exactly the subset of $\mc{V}_n$ such that 
$L_{\overline{\varphi}}$ is not isomorphic to $F(M) \otimes \mu$ where
$\mu$ is a fixed one dimensional module.
The proof is concluded by observing that
$$
\mc{U} = \bigcap_\mu \mc{U}_\mu
$$
where the intersection is over all one dimensional modules $\mu$.  Thus,
$\mc{U}$ is closed.
\end{proof}

With the preceding results established, the main result may now be proven.

\begin{proof}[Proof of Theorem \ref{T: fin many endo of dim n}]
Let $M$ be an indecomposable endotrivial module in $\mc{F}_{\rel{g}}$ of dimension
$n$.  Let $\mc{U}_M$ be the subset of $\mc{V}_n$ of representations
$\ov\varphi$ such that the associated module $L_{\overline{\varphi}}$ is isomorphic
to $F(M) \otimes \lambda$ for some one dimensional module $\lambda$.

By the previous lemma, $\mc{U}_M$ is open in $\mc{V}_n$ and so
$\overline{\mc{U}_M}$ is a union of components in $\mc{V}_n$.
There are finitely many components of $\mc{V}_n$, and any
endotrivial module will be contained in some $\mc{U}_M$ for some
endotrivial module $M$.
Since $A_{n}$ has only finitely many isomorphism classes of one
dimensional modules,
there are finitely many isomorphism classes in each such $\mc{U}_M$.
Hence, we conclude
that there are only finitely many endotrivial modules in $\mc{F}_{\rel{g}}$ of
dimension $n$.
\end{proof}

\subsection{Conditions for Finiteness} \label{S: conditions on g}

Now that the main theorem of the chapter has been established, we
consider conditions on a classical Lie superalgebra $\mf{g}$ which imply that there
are finitely many simple modules in $\mc{F}$ of dimension $\leq n$.

One such case to guarantee this is when
$\ev{\mf{g}}$ has finitely many simple modules
of dimension $\leq n$ since the condition on $\ev{\mf{g}}$ can be extended
to all of $\mf{g}$.

\begin{lemma} \label{L: finite even finite full}
Let $\mf{g} = \supalg{\mf{g}}$ be a classical Lie superalgebra such that $\ev{\mf{g}}$
has finitely many simple modules of dimension $\leq n$.  Then
$\mc{F}_{\rel{g}}$ has finitely many simple modules of dimension $\leq n$.
\end{lemma}
\begin{proof}
Let $S$ be a simple module in $\mc{F}$ of dimension $\leq n$.  Then
$S|_{\ev{\mf{g}}}$ has a simple $\ev{\mf{g}}$-module in its socle, call
it $T$.  Then,
$$
0 \neq \Hom_{U(\ev{\mf{g}})}(T, S|_{\ev{\mf{g}}}) =
\Hom_{\U{g}}(\ind{g} T, S)
$$
and so there is a surjection $\ind{g} T \twoheadrightarrow S$ for some
simple $\ev{\mf{g}}$-module $T$.

Consider the set
$$
\ms{C} = \{ \ind{g} T \st T \text{ is a simple $\ev{\mf{g}}$-module of
$\dim \leq n$} \}
$$
Note that each element of $\ms{C}$ is finite dimensional since $T$ is
finite dimensional.  Furthermore,
by the previous observation, a module from this set surjects onto any
simple module in $\mc{F}$, and since the set $\ms{C}$ is finite by assumption,
the result is proven.
\end{proof}

Now the Lie algebras which satisfy this condition are considered.
Since $\mf{g} = \supalg{\mf{g}}$ is classical, $\ev{\mf{g}}$ is reductive.
Note that if $\ev{\mf{g}}$ has any central elements, then there will be
infinitely many one dimensional modules where the central elements
act via scalars.  Thus, it must be that $\ev{\mf{g}}$ is semisimple.

Then $\ev{\mf{g}} \cong \mf{h}_1 \times \dots \times \mf{h}_s$ where
$h_i$ is a simple Lie algebra.  Then if $L(\lambda)$ is any simple finite dimensional
$\ev{\mf{g}}$-module, $L(\lambda) \cong L(\lambda_1) \boxtimes \dots
\boxtimes L(\lambda_s)$ where $L(\lambda_i)$ is a simple $\mf{h}_i$
module.

Since there are finitely many weights $\lambda_i$ of $X(\mf{h}_i)$ such
that $L(\lambda_i)$ is of dimension $\leq n$, then the same result
holds for $L(\lambda)$ and consequently $\ev{\mf{g}}$.

\begin{corollary}
Let $\mf{g}$ be a classical Lie superalgebra such that $\ev{\mf{g}}$
is a semisimple Lie algebra.  Then there are finitely many isomorphism
classes of endotrivial
modules in $\mc{F}_{\rel{g}}$ of dimension $n$ for any $n \in \Z$.
\end{corollary}
\begin{proof}
Since $\ev{\mf{g}}$ is semisimple, there are only finitely many simple
modules of dimension $\leq n$ for $n \in \Z$.  The result follows
from Lemma \ref{L: finite even finite full} and
Theorem \ref{T: fin many endo of dim n}.
\end{proof}

The conditions given in Theorem \ref{T: fin many endo of dim n} are
sufficient but not necessary for all classical Lie superalgebras.  Some
interesting cases are Lie superalgebras whose even component contains
a torus, and thus are not semisimple.
In this case, it may be possible to conclude the result
assuming that there are only finitely many \emph{one} dimensional modules,
regardless of what $n$ may be.  The condition of having only finitely
many one dimensional representations is explored here.

For the detecting subalgebras, we have seen by direct computation that there are only
finitely many endotrivial modules of a fixed dimension $n$.  For these Lie superalgebras,
there are only
\emph{two} one dimensional modules, $k_{ev}$ and $k_{od}$.

Now consider an example where this condition fails.  When $\mf{g} =
\mf{gl}(1|1)$, there are infinitely many one dimensional $\mf{g}$-modules.
The matrix realization $\mf{gl}(1|1)$ has basis vectors $x$ and $y$ as
in $\sL$, but has two toral basis elements $t_1$ and $t_2$ which are
given by
\[ t_1 = \left( \begin{array}{cc}
		1 & 0 \\
		0 & 0
		\end{array} \right)	\quad	
		t_2 =			
		\left( \begin{array}{cc}
		0 & 0 \\
		0 & 1
		\end{array} \right) \]
and the weights of the simple modules in $\mf{g}$
are given by $(\lambda| \mu)$ where
$\lambda, \mu \in k$ and $t_1$ and $t_2$ act on $k$ via multiplication
by $\lambda$ and $\mu$ respectively.  If $\lambda = -\mu$ then the
representation of the simple $\mf{g}$-module is one dimensional.  Thus
there are infinitely many one dimensional modules
given by the representations $(\lambda|-\lambda)$.

In general, the condition that there are
finitely many one dimensional modules in $\mc{F}_{\rel{g}}$
is equivalent to the condition that
$\ev{\mf{g}}/([\mf{g},\mf{g}] \cap \ev{\mf{g}})$ has finitely many one dimensional
modules.

\begin{proposition} \label{P: finitely many one dim condition}
Let $\mf{g}$ be a classical Lie superalgebra.  Then there are finitely
many one dimensional modules in $\mc{F}_{\rel{g}}$ if and only if
$\ev{\mf{g}} \subseteq [\mf{g},\mf{g}]$.
\end{proposition}
\begin{proof}
Assume that $\mf{g}$ is a Lie superalgebra such that
$\ev{\mf{g}} \subseteq [\mf{g},\mf{g}]$
A one dimensional
representation in $\mc{F}$ corresponds to
a Lie superalgebra homomorphism $\varphi: \mf{g} \rightarrow k_{ev}$ since
$\End_k(k)$ (for $k$ either even or odd) is always isomorphic to $k_{ev}$.
Since $k_{ev}$ is concentrated in degree $\ov{0}$ and
$\varphi$ is an even map, any element of $\od{\mf{g}}$ necessarily maps to 0.
Furthermore, since $k_{ev}$ is abelian as a Lie superalgebra, then
$[\mf{g},\mf{g}]$ must be mapped to zero and so by assumption $\ev{\mf{g}}$
maps to 0 as well and $\varphi$ is the zero map.  This forces the one dimensional
module to be either $k_{ev}$ or $k_{od}$.

Now, assume that $\mf{g}$ has only finitely many one dimensional modules
but that $\ev{\mf{g}} \nsubseteq [\mf{g},\mf{g}]$.  Let
$g \in \ev{\mf{g}} \setminus [\mf{g},\mf{g}]$.  As noted, since
$k_{ev}$ is abelian, if $\varphi$ is a representation of $\mf{g}$,
then $\ov{\varphi} : \ev{\mf{g}} / ([\mf{g}, \mf{g}] \cap \ev{\mf{g}})
\rightarrow k_{ev}$ yields
another representation which agrees on nonzero elements.  They are also
equivalent in the sense $\varphi$ can be obtained uniquely from $\ov{\varphi}$
and vice versa.
Since $\mf{g}$ is classical, $\ev{\mf{g}}$ is a reductive Lie algebra and
given that $g \notin [\mf{g}, \mf{g}]$, $g$ must be in the center of
$\ev{\mf{g}}$ and
therefore in the torus of $\ev{\mf{g}}$ as well.  Thus, $g$ is a semisimple
element, and in $\mc{F}$, $g$ must act diagonally on any one dimensional module.  If $\ov{g}$ is the image of $g$ in
$\ev{\mf{g}} / ([\mf{g}, \mf{g}] \cap \ev{\mf{g}})$, then $\ov{g}$ is
nonzero and $\langle \ov{g} \rangle$ is a one dimensional abelian
Lie superalgebra.  Since $\ov{g}$ acts diagonally and
$k$ is infinite, this yields infinitely many distinct one dimensional
modules resulting from the diagonal action of $\ov{g}$.  These one
dimensional modules lift (possibly non-uniquely) to
$\ev{\mf{g}} / ([\mf{g}, \mf{g}] \cap \ev{\mf{g}})$ and consequently
$\mf{g}$ as well.  This is a contradiction
and the assumption that
$\ev{\mf{g}} \nsubseteq [\mf{g},\mf{g}]$ must be false.
\end{proof}
\begin{corollary}
Let $\mf{g}$ be a simple classical Lie superalgebra.  Then there are finitely
many one dimensional modules in $\mc{F}_{\rel{g}}$.
\end{corollary}
\begin{proof}
By the necessary condition of simplicity given in Proposition 1.2.7
of \cite{Kac-1977} that $[\od{\mf{g}},\od{\mf{g}}] = \ev{\mf{g}}$,
the lemma is proven.
\end{proof}
Note that in the case where there are finitely many one dimensional
modules in $\mc{F}$, there are in fact only two, $k_{ev}$ and $k_{od}$,
as in the case of the detecting subalgebras.

Having established Proposition \ref{P: finitely many one dim condition}, we again consider
the concrete example of the $\mf{g} = \mf{gl}(1|1)$.  Here, $[\mf{g},\mf{g}] =
\langle x,y, t_1 + t_2 \rangle$ and so
$\ev{\mf{g}}/ ([\mf{g},\mf{g}] \cap \ev{\mf{g}})
= \langle t_1, t_2 \rangle / (t_1 + t_2)$.  Then, it is clear that if
$t_1$ has any weight $\lambda$, then $\mu$ is determined
to be $-\lambda$.  Thus, there are infinitely many one dimensional
representations resulting from the free parameter $\lambda$.

\section*{Acknowledgments}
The author would like to thank his Ph.D. thesis advisor
Daniel Nakano for his
guidance in both the formulation of questions as well as development
of the tools needed for studying endotrivial supermodules.
Additional thanks are due to Jon Carlson and Jonathan Kujawa for their
helpful comments
and suggestions that have been incorporated into the current version
of this paper.  Finally, the author
would like to thank the referee for their careful and thorough review
which greatly improved the quality of this work.


\begin{thebibliography}{99}

	\bibitem{Benson1-1998}
	David J. Benson,
	\emph{Representations and cohomology. I},
	second ed.,
	Cambridge Studies in Advanced Mathematics, Vol. 30,
	Cambridge University Press,
	Cambridge,
	1998.

	\bibitem{BKN1-2006}
	Brian D. Boe, Jonathan R. Kujawa, Daniel K. Nakano,
	Cohomology and support varieties for Lie superalgebras,
	\emph{Trans. Amer. Math. Soc.}, \textsf{362} (2010), 6551-6590.

	\bibitem{BKN3-2009}
	Brian D. Boe, Jonathan R. Kujawa, Daniel K. Nakano,
	Complexity and module varieties for classical Lie superalgebras,
	\emph{Int. Math. Res. Not.}, (3) \textsf{2011} (2011), 696-724.
	
	\bibitem{Brundan-2002}
	Jonathan Brundan,
	Kazhdan-Lusztig polynomials and character formulae for the
			Lie superalgebra $\mf{q}(n)$,
	Adv. Math. (1) \textsf{182} (2004), 28-77.
	
	\bibitem{BS-2012}
	Jonathan Brundan, Catharina Stroppel,
	Highest weight categories arising from Khovanov's diagram algebra IV: the general linear supergroup,
	\emph{J. Eur. Math. Soc.},
	(2) \textsf{14} (2012), 373–419.
	
	\bibitem{BK-2002}
	Jonathan Brundan, Alexander Kleshchev,
	Projective representations of symmetric groups
	via Sergeev duality,
	\emph{Math. Z.},
	(1) \textsf{239} (2002), 27-68.
	
	\bibitem{Carlson-1980}
	Jon F. Carlson,
	Endo-trivial modules over $(p,p)$-groups,
	\emph{Illinois J. Math.}, (2) \textsf{24} (1980), 287-295.
	
	\bibitem{CHM-2010}
	Jon F. Carlson, David J. Hemmer, Nadia Mazza,
	The group of endotrivial modules for the symmetric and
		alternating groups
	\emph{Proc. Edinburgh Math. Soc.}, \textsf{53} (2010), 83-95.
	
	\bibitem{CMN-2006}
	Jon F. Carlson, Nadia Mazza, Daniel K. Nakano,
	Endotrivial modules for finite groups of Lie type,
	\emph{J. Reine Angew. Math.}, \textsf{595} (2006), 284-306.
	
	\bibitem{CMN-2009}
	Jon F. Carlson, Nadia Mazza, Daniel K. Nakano,
	Endotrivial modules for the symmetric and alternating groups,
	\emph{Proc. Edinburgh Math. Soc.}, \textsf{52} (2009), 45-66.
	
	\bibitem{CN-2009}
	Jon F. Carlson, Daniel K. Nakano,
	Endotrivial modules for finite group schemes,
	\emph{J. Reine Angew. Math.}, \textsf{653} (2011), 149-178.
	
	\bibitem{CN-2011}
	Jon F. Carlson, Daniel K. Nakano,
	Endotrivial modules for finite groups schemes II,
	\emph{Bull. Inst. Math. Acad. Sinica}, (2) \textsf{7} (2012),
	271-289.
	
	\bibitem{CT-2004}
	Jon F. Carlson, Jacques Th\'evenaz,
	The classification of endotrivial modules,
	\emph{Invent. Math.}, \textsf{158} (2004), 389-411.
	
	\bibitem{CT-2005}
	Jon F. Carlson, Jacques Th\'evenaz,
	The classification of torsion endotrivial modules,
	\emph{Ann. Math.}, (2) \textsf{165} (2005), 823-883.
	
	\bibitem{Dade-1979}
	Everett C. Dade,
	\emph{Algebraically rigid modules},
	Representation Theory II,
	Lecture Notes in Mathematics
	Vol. 832,
	Springer-Verlag,
	Berlin Heidelberg New York,
	1980,  pp. 195-215.
	
	\bibitem{Dade1-1978}
	Everett C. Dade,
	Endo-permutation modules over $p$-groups I,
	\emph{Ann. Math.}, (3)  \textsf{107} (1978), 459-494.
	
	\bibitem{Dade2-1978}
	Everett C. Dade,
	Endo-permutation modules over $p$-groups II,
	\emph{Ann. Math.}, (2) \textsf{108} (1978), 317-346.
	
	\bibitem{DEN-2004}
	Stphen R. Doty, Karin Erdmann, Daniel K. Nakano,
	 Extensions of modules over Schur algebras, symmetric groups, and Hecke algebras,
	 \emph{Algebr. Represent. Theory} \textsf{7} (2004), 67-100.
	
	\bibitem{Kac-1977}
	Victor G. Kac,
	Lie superlagebras,
	\emph{Adv. Math.}, \textsf{26} (1977), 8-96.
	
	\bibitem{Kumar-2002}
	Shrawan Kumar,
	\emph{Kac-Moody groups, their flag varieties and representation theory},
	Progress in Mathematics, Vol. 204, Birkh\"auser Boston Inc.,
	Boston, MA,
	2002.

	\bibitem{Puig-1990}
	Lluis Puig,
	Affirmative answer to a question of Feit,
	\emph{J. Algebra} \textsf{131} (1990), 513-526.
	
	\bibitem{Scheunert-1979}
	Manfred Scheunert,
	\emph{The theory of Lie superalgebras: an introduction},
	Lecture Notes in Mathematics Vol. 716,
	Springer-Verlag,
	Berlin Heidelberg New York,
	1979.
	
	\bibitem{HTT-2013}
	The Univalent Foundations Program,
	\emph{Homotopy Type Theory: Univalent Foundations of Mathematics},
	http://homotopytypetheory.org/book,
	Institute for Advanced Study,
	2013.
	
	

\end{thebibliography}
\end{document}